\newtheorem{example}{Example}
\newtheorem{remark}{Remark}
\newtheorem{theorem}{Theorem}
\newtheorem{proof}{Proof}
\documentclass[preprint,12pt]{elsarticle}
\usepackage{algorithm}
\usepackage{algpseudocode}
\usepackage{subfigure}
\usepackage{adjustbox}

\usepackage{amssymb}
\usepackage{amsmath}

\journal{Nuclear Physics B}

\begin{document}

\begin{frontmatter}



\title{Memory-Augmented Phase-Field Modeling: Onsager Formulation, Internal-Variable Lifting, and Structure-Preserving Computation}

\author[1]{Peng Ding}
\ead{pding@eitech.edu.cn}

\author[2]{Qingxia Liu}
\ead{liuqx@xmu.edu.cn}

\author[1]{Zhiping Mao\corref{cor1}}
\ead{zmao@eitech.edu.cn}

\author[3]{Zhen Zhang}
\ead{zhangz@sustech.edu.cn}

\cortext[cor1]{Corresponding author.}
\affiliation[1]{organization={School of Mathematical Sciences, Eastern Institute of Technology, Ningbo},
            postcode={315200},
            state={Zhejiang},
            country={China}}
\affiliation[2]{organization={School of Mathematical Sciences, Xiamen University},
            city={Xiamen},
            postcode={361005},
            state={Fujian},
            country={China}}

\affiliation[3]{organization={Department of Mathematics, Southern University of Science and Technology (SUSTech)},
            city={Shenzhen},
            postcode={518055},
            state={Guangdong},
            country={China}}


\begin{abstract}
Classical phase-field models are typically formulated as Markovian gradient flows, where the instantaneous state completely determines the future evolution. However, many complex materials and physical processes exhibit relaxation behaviors with memory effects that cannot be accurately described by such local-in-time models. Existing approaches often introduce fractional-order operators as phenomenological extensions, but they may lack a rigorous thermodynamic interpretation and can obscure the intrinsic energy-dissipation structure of phase-field dynamics. In this work, we develop a general memory-augmented phase-field (MAPF) framework based on Onsager's variational principle, in which memory effects are incorporated through a general relaxation kernel while preserving the thermodynamic consistency of the underlying system. Rather than treating fractional models as fundamental formulations, we show that they can be interpreted as reduced descriptions associated with specific power-law memory kernels. The proposed MAPF equations provide a unified modeling framework for non-Markovian phase-field dynamics. To facilitate analysis and computation, we further introduce an equivalent internal-variable representation that transforms the original nonlocal-in-time dynamics into an extended Markovian gradient system. This formulation preserves key structural properties, including energy dissipation and mass conservation, and enables the development of efficient structure-preserving numerical schemes. In particular, we construct an energy-stable scalar auxiliary variable approach combined with a fast rank-one implementation that exploits the relaxation-spectrum-induced coupling structure, resulting in a computationally efficient algorithm with reduced storage requirements. The proposed framework establishes a systematic connection between memory-dependent phase-field modeling, internal-variable thermodynamics, and structure-preserving numerical computation. Numerical experiments demonstrate the accuracy, stability, and efficiency of the proposed methods for phase-field dynamics with long-time memory effects.
\end{abstract}




\end{frontmatter}



\section{Introduction}
Phase-field models are a standard energetic framework for describing interfacial dynamics, phase separation, and pattern formation. In their classical form, the Allen--Cahn and Cahn--Hilliard equations are local-in-time gradient flows driven by a free energy $E$ and a thermodynamic force $\mu=\delta E/\delta\phi$:
\begin{equation*}\label{eq:AC-CH}
    \frac{\partial \phi}{\partial t} = -\mathcal{G}\mu, \qquad \mu=\frac{\delta E}{\delta\phi},
\end{equation*}
where $\mathcal{G}=I$ gives Allen--Cahn dynamics and $\mathcal{G}=-\Delta$ gives Cahn--Hilliard dynamics. This formulation is thermodynamically natural: under suitable boundary conditions it yields the classical energy-dissipation law
\begin{equation*}
    \frac{d}{dt}E(\phi(t)) = -\|\mathcal{G}^{1/2}\mu(\mathbf{x},t)\|^2 \le 0 .
\end{equation*}
It is also Markovian: the present rate $\frac{\partial \phi(\mathbf{x},t)}{\partial t}$ depends only on the instantaneous driving force. Such a closure is appropriate when unresolved microscopic or mesoscopic variables relax much faster than the order parameter, and it underlies a large body of thermodynamically consistent phase-field theory and computation~\cite{AllenCahn1979,CahnHilliard1958,Gurtin1996Microforce,QiangDu,ShenYang2010,SHEN2019474}.

For many materials of current interest, however, the instantaneous closure is not the most faithful modeling assumption. Phase separation in polymer solutions, gels, crowded cytoplasm, protein condensates, and other soft or biological media may be coupled to polymer-chain relaxation, transient crosslinks, cytoskeletal elasticity, binding--unbinding kinetics, or glassy structural rearrangements. These mechanisms may relax on time scales comparable to the phase-separation process. The resulting dynamics is then not determined solely by the present chemical potential; it retains a memory of past deformation, composition change, or internal rearrangement. This viewpoint is consistent with viscoelastic phase separation and related biological applications~\cite{Tanaka2000VPS,Tanaka2022BioCells,CurkLuijten2023PNAS}, and with coarse-grained modeling principles showing that memory kernels become important when the resolved and unresolved variables are not clearly separated in time~\cite{ColemanGurtin1967InternalVariables,LiBianLiKarniadakis2015MZ}.

A modeling-oriented phase-field equation with memory may be written, at the macroscopic level, as
\begin{equation}\label{eq:intro_memory_kernel_model}
    \int_0^t K(t-s)\,\frac{\partial \phi}{\partial s}(\mathbf{x},s)\,ds = -\mathcal{G}\mu(\mathbf{x},t), \qquad \mu(\mathbf{x},t)=\frac{\delta E}{\delta\phi}(\mathbf{x},t),
\end{equation}
where $K(t)$ is a relaxation kernel. The kernel is not merely a mathematical decoration: it represents the effect of 
internal modes that have been eliminated from the macroscopic description. 
Different choices of $K(t)$ encode different physical assumptions. 
A single exponential kernel corresponds to one Maxwell relaxation mode; 
a finite sum of exponentials corresponds to a generalized Maxwell or Prony-series 
representation~\cite{Ferry1980,Tschoegl1989};
a broad relaxation spectrum produces non-exponential memory; 
and a power-law spectrum yields the Caputo time-fractional derivative 
\cite{Mainardi2010Fractional,Bagley1983Fractional}. Earlier phase-field equations with memory already showed that relaxation kernels can generate hyperbolic or hereditary generalizations of classical phase-field models~\cite{RotsteinBrandonNovickCohenNepomnyashchy2001Memory}. The central issue is therefore not whether a model is integer-order or fractional, but how the memory kernel is generated and how the stored memory energy dissipates.

This distinction is particularly important for time-fractional phase-field equations. A common construction is to start from the classical Allen--Cahn or Cahn--Hilliard gradient flow and replace the first-order time derivative by a Caputo fractional derivative. This produces a useful class of fractional generalizations for anomalous transport, slow coarsening, and material memory~\cite{LiuChengWangZhao2018TFPF,HouXu2021,DuYangZhou2020JSC,AlMaskariKaraa2021IMAJNA}. From a strict modeling perspective, however, such a replacement does not by itself identify the internal variables that were eliminated, the energy temporarily stored in memory, or the friction mechanism responsible for dissipation. Consequently, the classical phase-field energy law is no longer automatic. Indeed, the energy-dissipation structure of time-fractional phase-field equations has required separate analysis, including integral-type dissipation laws, nonlocal modified energies, and decreasing upper bounds for the physical energy~\cite{Tang_PF,QuanTangYang2020,LiQuanXu2022Energy,QuanTangYang2023}. These results are mathematically important, and they also indicate a modeling gap: the thermodynamic structure should ideally be built into the model before discretization.

Motivated by these limitations, we aim to develop in this work a thermodynamically consistent framework for phase-field modeling with memory effects. In particular, the viewpoint of this paper is to reverse the usual order of model construction. Rather than taking the fractional derivative as the primary modeling ingredient, we begin with an enlarged system characterized by a continuum of relaxation modes. The resulting framework explicitly reveals the internal memory variables, the associated relaxation mechanisms, and the stored memory energy.

Specifically, we consider memory effects from a more fundamental viewpoint based on relaxation mechanisms. By incorporating a general relaxation spectrum into the Onsager variational framework, we establish a new class of memory-augmented phase-field (MAPF) equations. The proposed framework provides a unified description of non-Markovian phase-field dynamics while retaining the intrinsic gradient-flow structure of classical models.To this end, 
the order parameter is coupled to a continuum of internal Maxwell-type modes $\Psi(\xi)$ indexed by a relaxation rate $\xi>0$.
The relaxation spectrum $\kappa(\xi)$ determines the memory kernel
\begin{equation*}
    K(t)=\int_0^\infty \kappa(\xi)e^{-\xi t}\,d\xi .
\end{equation*}
For any positive spectrum, Onsager's variational principle gives an augmented free energy and a nonnegative dissipation potential. Eliminating the internal variables produces the memory equation \eqref{eq:intro_memory_kernel_model}. If $\kappa(\xi)\propto \xi^{\alpha-1}$, then $K(t)=t^{-\alpha}/\Gamma(1-\alpha)$ and the reduced model is precisely a Caputo time-fractional phase-field equation. Thus the fractional model is not the starting point; it is one particular reduced model generated by a power-law relaxation spectrum. Other spectra or nonlinear internal energies lead, within the same framework, to other phase-field models with finite-memory, tempered, distributed-order, Prony-series, stretched-exponential, or nonlinear memory effects \cite{sabzikar2015tempered,song2023nonmaxwellian,mainardi2008time,luo2018development}. This broader interpretation is consistent with internal-variable thermodynamics~\cite{ColemanGurtin1967InternalVariables,Ottinger2005BET}, power-law rheology in gels and living cells~\cite{WinterChambon1986,Fabry2003,Bonfanti2020}, and diffusive or infinite-state representations of hereditary operators~\cite{Montseny1998,diethelm2022new}.

A key feature of the proposed formulation is that the memory contribution is represented through a continuous spectrum of relaxation processes. This representation naturally connects the macroscopic memory kernel with internal variables associated with different relaxation time scales.
The internal-variable formulation provides a convenient framework for numerical approximation of the proposed MAPF equations. By replacing the history-dependent memory term with a family of auxiliary relaxation variables, the resulting system can be discretized using standard techniques for coupled evolution equations. Nevertheless, the continuous relaxation spectrum introduced by the memory kernel may still lead to additional computational costs, particularly for long-time simulations. This motivates the development of efficient numerical strategies that can accurately capture the memory effects while maintaining computational feasibility.
To this end,  we derive an equivalent internal-variable lifting for the resulting MAPF equations, thereby transforming the original nonlocal-in-time dynamics into an extended Markovian gradient system. This reformulation is not only beneficial for theoretical analysis but also provides a practical foundation for developing efficient numerical methods.

The main contributions of this work are summarized below:
\begin{itemize}

\item \textbf{A memory-based phase-field modeling framework.} We introduce a new class of memory-augmented phase-field (MAPF) models derived from Onsager's variational principle. The framework incorporates general relaxation kernels into the thermodynamic driving force and provides a unified description of phase-field systems with memory effects. In particular, fractional phase-field models are interpreted as special cases associated with specific power-law memory kernels, rather than independent fundamental models.

\item \textbf{An internal-variable representation of non-Markovian phase-field dynamics.} We establish an equivalent internal-variable formulation of the MAPF equations by introducing a family of relaxation variables parameterized by the characteristic time scale. This lifting transforms the original history-dependent system into an extended Markovian gradient-flow system while preserving the underlying thermodynamic structure. The formulation provides a systematic bridge between memory-dependent models and classical phase-field dynamics.

\item \textbf{Structural analysis of the proposed MAPF system.} We rigorously establish the fundamental properties of the proposed formulation, including energy dissipation for Allen--Cahn type models and mass conservation together with energy decay for Cahn--Hilliard type models. These results demonstrate that the memory augmentation is compatible with the fundamental variational structure of phase-field equations.

\item \textbf{Structure-preserving discretization and efficient implementation of MAPF dynamics.} Based on the internal-variable representation, we develop energy-stable exponential scalar auxiliary variable (ESAV) schemes for the MAPF equations. To overcome the computational challenges caused by the continuous relaxation spectrum, fast low-rank approximations are introduced to reduce computational complexity and storage requirements while maintaining the stability properties of the original formulation.

\item \textbf{Numerical validation of long-time memory dynamics.} 
Extensive numerical experiments are performed to verify the accuracy, stability, and efficiency of the proposed methods. The results demonstrate that the proposed MAPF framework can effectively capture long-time relaxation behaviors while preserving the essential physical structures of phase-field dynamics.
\end{itemize}

The remainder of this paper is organized as follows. In Section~\ref{sec:derivation}, we establish the MAPF framework based on Onsager's variational principle and recover the Caputo phase-field model as a power-law special case. Section~\ref{sec:2} develops the equivalent internal-variable formulation of the MAPF and analyzes its continuous energy dissipation and mass-conservation properties. Section~\ref{sec:3} presents the structure-preserving ESAV schemes, the fast rank-one implementation, and the corresponding discrete energy stability and mass-conservation properties. Numerical experiments are presented in Section~\ref{sec:6}, followed by conclusions in Section~\ref{sec:7}.

\section{Variational derivation of memory-augmented phase-field models}\label{sec:derivation}
We now derive a class of Allen--Cahn phase-field models with memory from Onsager's variational principle (OVP). The purpose of this section is to formulate the modeling structure before choosing any particular kernel. The key idea is to retain the slow internal relaxation modes that are eliminated in a Markovian closure. Coupling the order parameter to a \emph{continuum of internal elastic modes} produces a hereditary phase-field equation whose memory kernel is determined by the relaxation spectrum. A power-law spectrum will later recover the Caputo time-fractional model as a special reduced case, but the derivation itself applies to any positive relaxation spectrum. We first present the physical picture (\S\ref{subsec:physical}), then carry out the variational derivation (\S\ref{subsec:OVP}), and finally establish the energy dissipation law (\S\ref{subsec:EDL}).

\subsection{Physical setup: the generalized Maxwell model}\label{subsec:physical}

Consider phase separation occurring in a viscoelastic medium --- for instance, a polymer gel, a biomolecular condensate, or a cytoskeletal network. These systems can be described by a binary mixture (polymer and solvent) in the
poor-solvent regime, where the equilibrium state is two coexisting phases of different
polymer volume fractions. The polymer chains are connected by transient cross-links --- physical entanglements, multivalent protein--protein binding, or hydrogen-bond
networks --- forming an elastic network that permeates the solution. The local polymer volume fraction is given by an order parameter $\phi(\mathbf{x},t)$.
In a purely
viscous solvent (no network), the dynamics would be a standard Cahn--Hilliard or
Allen--Cahn equation.  However, the polymer network introduces two new effects:
\begin{enumerate}
  \item \textbf{Elastic energy storage:} When phase separation drives polymer to
    accumulate locally, the network is deformed (stretched in the polymer-rich phase,
    compressed in the polymer-poor phase), with the local elastic deformation represented by an internal variable $\Psi(\mathbf{x}, t, \xi)$ (mesh-size change relative to the reference state) of network modes with characteristic relaxation rate~$\xi\in(0,\infty)$.  A ``mode with rate~$\xi$'' corresponds to
cross-links of lifetime~$1/\xi$: short-lived bonds (large~$\xi$) relax quickly and
offer transient resistance, while long-lived bonds (small~$\xi$) sustain elastic stress
over extended periods.  
This network deformation stores elastic energy that
    partially opposes the thermodynamic driving force. In this work, we assume the elastic energy is of Hookean type, i.e., $\frac{1}{2}\kappa(\xi)|\Psi|^2$, with $\kappa(\xi)$ being the modulus. 
  \item \textbf{Permeation friction:} For the polymer volume fraction $\phi$ to
    change at a given point, polymer chains must move \emph{relative to} the elastic
    network.  This relative motion --- polymer ``flowing through'' its own network --- is
    resisted by friction, analogous to solvent permeating through a gel. If the network deforms in perfect synchrony with the concentration change
    ($\dot\Psi = \dot\phi$), the polymer and the network move together as a unit ---
    there is no relative motion and hence no friction. If $\dot\phi \neq \dot\Psi$, polymer must flow through the network (or the
    network must be dragged through the polymer), producing a frictional force
    proportional to the velocity mismatch. The stronger
    the mismatch between the rate of concentration change ($\dot\phi$) and the rate
    of network deformation ($\dot\Psi$), the greater the frictional dissipation -- known as \emph{permeation friction}. This dissipation is assumed to be in the quadratic form $\frac{\kappa(\xi)}{\xi}|\dot\phi - \dot\Psi|^2$.
\end{enumerate}

This is precisely the scalar analogue of the two-fluid friction
$\zeta |\mathbf{v}_{\mathrm{polymer}} - \mathbf{v}_{\mathrm{solvent}}|^2$ in the
Doi--Onuki and Tanaka models of viscoelastic
phase separation \cite{DoiOnuki1992, Tanaka1997, Qian2026Incorporating}.  The mode viscosity $\kappa(\xi)/\xi$ plays the role of the friction
coefficient~$\zeta$ restricted to the sub-network with relaxation rate~$\xi$.
Then the present model can be understood as a simplified (scalar, linearized,
continuous-spectrum) version of the full two-fluid model for viscoelastic phase
separation. This relation is summarized in Table \ref{table:relation}.
\begin{table}[!htbp]
\label{table:relation}
\caption{Relation to the Doi--Onuki / Tanaka two-fluid model.}
\begin{center}
\renewcommand{\arraystretch}{1.3}
\begin{tabular}{lll}
\hline
\textbf{Simplification} & \textbf{Two-fluid model} & \textbf{Present model} \\
\hline
Kinematics & Vector velocities $\mathbf{v}_p$, $\mathbf{v}_s$
  & Scalar rates $\dot\phi$, $\dot\Psi$ \\
Elastic stress & Nonlinear tensor $\boldsymbol{\sigma}_{\mathrm{net}}$
  & Linear scalar $\kappa(\xi)\Psi$ \\
Relaxation spectrum & Single time $\lambda$ (or discrete)
  & Continuous $\xi \in (0,\infty)$ \\
Friction & $\zeta(\varphi)|\mathbf{v}_p - \mathbf{v}_s|^2$
  & $\frac{\kappa(\xi)}{\xi}|\dot\phi - \dot\Psi|^2$ \\
\hline
\end{tabular}
\end{center}
\end{table}

\paragraph{The mechanical analogue}
The physical picture described above is formalized by representing each network mode
as a Maxwell element (a spring in series with a dashpot), with the order
parameter~$\phi$ driving the system in parallel (see Fig.~\ref{fig:maxwell}).
Specifically:
\begin{itemize}
    \item \textbf{Spring:} modulus $\kappa(\xi)$, stores energy $\frac{1}{2}\kappa(\xi)|\Psi|^2$;
    \item \textbf{Dashpot:} viscosity $\eta(\xi) = \kappa(\xi)/\xi$, dissipates energy at rate $\frac{\kappa(\xi)}{\xi}|\frac{\partial \phi}{\partial t} - \frac{\partial \Psi}{\partial t}|^2$;
    \item \textbf{Relaxation rate:} $\xi = \kappa/\eta$, so the relaxation time is $1/\xi$.
\end{itemize}
The full system is a \emph{generalized Maxwell model}: the order parameter $\phi$ is coupled in parallel to a continuum of such elements (see Figure~\ref{fig:maxwell}). As we shall see, with initial condition 
\begin{equation}\label{eq:Psi0}
    \Psi(\mathbf{x}, 0, \xi) = 0 \quad \text{(undeformed reference state),}
\end{equation}
this generalized Maxwell model yields
an explicit solution
\begin{equation}\label{eq:Psi_solution_preview}
  \Psi(\mathbf{x}, t, \xi) = \int_0^t e^{-\xi(t-s)} \frac{\partial\phi}{\partial s}(\mathbf{x},s)\,ds,
\end{equation}
i.e., an exponentially-weighted memory of past concentration changes.  Physically,
$\Psi(\mathbf{x},t, \xi)$ is the cumulative elastic strain in bonds of lifetime~$1/\xi$ that has
\emph{not yet relaxed}: recent changes in~$\phi$ contribute fully, while distant
changes have been progressively dissipated through the dashpot.

\begin{figure}[ht]
\centering
\includegraphics[width=0.85\linewidth]{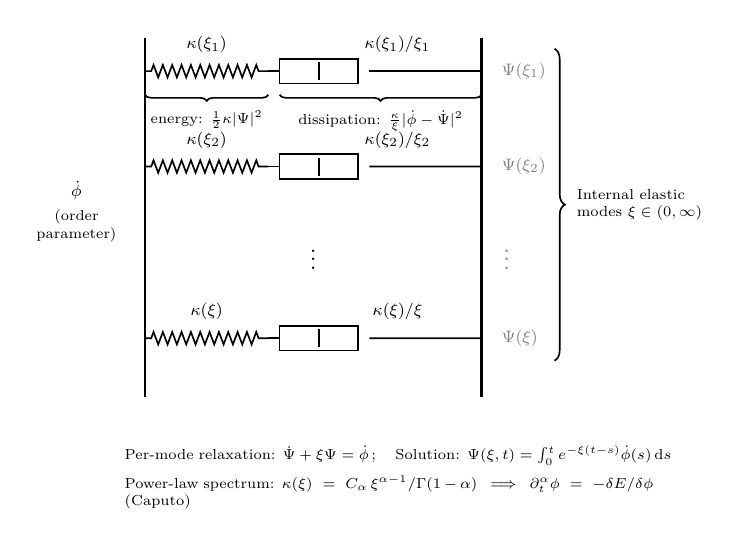}
\caption{The generalized Maxwell model underlying the phase-field equation with memory. The macroscopic field $\phi$ drives a continuum of spring--dashpot elements. General relaxation spectra yield memory kernels, and the power-law spectrum $\kappa(\xi)\propto \xi^{\alpha-1}$ yields the Caputo time-fractional special case.}\label{fig:maxwell}
\end{figure}

\paragraph{Power-law spectra and fractional reductions}
A power-law spectrum is an important, but not exclusive, constitutive choice. If
\begin{equation}\label{eq:power_law_spectrum}
    \kappa(\xi) = C_\alpha\, \xi^{\alpha-1}, \quad 0<\alpha<1,
\end{equation}
then the relaxation modulus becomes $K(t) = \int_0^\infty \kappa(\xi)e^{-\xi t}\,d\xi \propto t^{-\alpha}$. This power-law rheology is the continuum-relaxation interpretation of the Scott--Blair fractional element and has been observed in:
\begin{itemize}
    \item critical gels at the gel point, where $K(t)\propto t^{-\alpha}$ with exponents close to $1/2$ in classical experiments by Winter and Chambon~\cite{WinterChambon1986};
    \item living cells and cytoskeletal materials, where power-law rheology with broad exponent ranges is commonly reported~\cite{Fabry2003,Bonfanti2020};
    \item biomolecular condensates and other multivalent soft materials, where broad relaxation spectra can arise from many binding and rearrangement pathways~\cite{Tanaka2022BioCells,CurkLuijten2023PNAS}.
\end{itemize}
In the context of biomolecular condensates and cytoskeletal networks, the power-law
spectrum reflects the \emph{hierarchical architecture} of the transient network: bonds
of all lifetimes --- from fast-exchanging weak interactions to slowly-dissociating
multivalent hubs --- contribute to a scale-free relaxation landscape.  The fractional
exponent~$\alpha$ thus encodes the breadth of the cross-link lifetime distribution: as
$\alpha \to 1$, the spectrum sharpens to a single relaxation time (pure Maxwell
element, Markovian limit); as $\alpha \to 0$, the spectrum becomes maximally broad
(nearly elastic, strong memory). Theoretically, the power-law spectrum  can be derived rigorously for Rouse dynamics on percolation clusters at the gel point, where the fractal geometry of the network eliminates any characteristic length or time scale \cite{deGennes1976, Muthukumar1989}.
In the present paper, the power-law spectrum is used to connect the general memory model with the existing time-fractional phase-field literature.

    \subsection{Onsager's variational principle formulation}\label{subsec:OVP}
    
    We now formalize the physical picture above within the framework of Onsager's variational principle.
    We assume that \(\kappa(\xi)\ge 0\) and that the corresponding memory kernel \(K(t)=\int_0^\infty \kappa(\xi)e^{-\xi t}\,d\xi\) is finite for \(t>0\), with the internal variables belonging to the natural weighted space such that the energy and dissipation integrals below are finite.

\subsubsection{Free energy}

In the extended state space $(\phi, \Psi)$, the total free energy consists of the classical Ginzburg--Landau energy for $\phi$ plus the elastic energy stored in the internal modes:
\begin{equation}\label{eq:free_energy}
    \mathcal{F}(\phi, \Psi) = \underbrace{\int_\Omega H(\phi, \nabla\phi)\,d\mathbf{x}}_{E(\phi):\text{ phase-field energy}}
    + \underbrace{\frac{1}{2}\int_\Omega\int_0^\infty \kappa(\xi)\,\Psi^2(\mathbf{x},t,\xi)\,d\xi\,d\mathbf{x}}_{\text{internal elastic energy}}.
\end{equation}
The first term drives phase separation (e.g., $H(\phi, \nabla\phi) = \frac{\varepsilon^2}{2}|\nabla\phi|^2 + F(\phi)$ with double-well potential $F$). The second term represents energy temporarily stored in the partially relaxed internal modes.

\subsubsection{Dissipation potential}

The dissipation arises from the friction due to the \emph{rate mismatch} between the macroscopic evolution $\dot\phi$ and the internal adaptation $\dot\Psi(\xi)$:
\begin{equation}\label{eq:dissipation}
    \Phi\!\left(\frac{\partial\phi}{\partial t}, \frac{\partial\Psi}{\partial t}\right)
    = \frac{1}{2}\int_\Omega\int_0^\infty \frac{\kappa(\xi)}{\xi}
    \left(\frac{\partial\phi}{\partial t} - \frac{\partial\Psi}{\partial t}\right)^{\!2} d\xi\,d\mathbf{x}.
\end{equation}
Physically, $\kappa(\xi)/\xi$ is the mode viscosity (dashpot coefficient). The dissipation is entirely due to internal friction --- there is no ``solvent'' contribution in the pure fractional model.

\subsubsection{The Rayleighian and Euler--Lagrange equations}

Following Onsager's variational principle, we define the Rayleighian
\begin{equation*}
    \mathcal{R}\!\left(\frac{\partial\phi}{\partial t}, \frac{\partial\Psi}{\partial t}\right)
    = \frac{d\mathcal{F}}{dt}\bigg|_{\dot\phi, \dot\Psi}
    + \Phi\!\left(\frac{\partial\phi}{\partial t}, \frac{\partial\Psi}{\partial t}\right),
\end{equation*}
where the rate of change of free energy is
\begin{equation}\label{eq:dFdt}
    \frac{d\mathcal{F}}{dt} = \int_\Omega\left[
    \frac{\delta E}{\delta\phi}\frac{\partial\phi}{\partial t}
    + \int_0^\infty \kappa(\xi)\,\Psi\,\frac{\partial\Psi}{\partial t}\,d\xi
    \right]d\mathbf{x}.
\end{equation}

\paragraph{Minimization $\mathcal{R}$ with respect to $\partial\Psi/\partial t$ (internal force balance)}
Setting $\delta\mathcal{R}/\delta\dot\Psi = 0$ yields, for each mode $\xi$:
\begin{equation}\label{eq:internal_balance}
    \underbrace{\kappa(\xi)\,\Psi(\mathbf{x},t,\xi)}_{\text{elastic restoring force}}
    = \underbrace{\frac{\kappa(\xi)}{\xi}\left(\frac{\partial\phi (\mathbf{x},t)}{\partial t} - \frac{\partial\Psi (\mathbf{x},t,\xi)}{\partial t}\right)}_{\text{friction force}}.
\end{equation}
This is precisely the force balance of a Maxwell element: the spring force equals the dashpot force. Rearranging gives the \emph{relaxation equation}:
\begin{equation}\label{eq:relaxation}
    \frac{\partial\Psi}{\partial t}(\mathbf{x},t,\xi) + \xi\,\Psi(\mathbf{x},t,\xi) = \frac{\partial\phi}{\partial t}(\mathbf{x},t),
    \qquad \forall\;\kappa(\xi) > 0.
\end{equation}
With initial condition $\Psi(\mathbf{x},0,\xi)=0$ (cf. \eqref{eq:Psi0}), the solution is given by \eqref{eq:Psi_solution_preview}.

\paragraph{Minimization $\mathcal{R}$ with respect to $\partial\phi/\partial t$ (macroscopic equation)}
Setting $\delta\mathcal{R}/\delta\dot\phi = 0$ gives:
\begin{equation*}\label{eq:macro_balance}
    \frac{\delta E}{\delta\phi} + \int_0^\infty \frac{\kappa(\xi)}{\xi}
    \left(\frac{\partial\phi (\mathbf{x},t)}{\partial t} - \frac{\partial\Psi (\mathbf{x},t,\xi)}{\partial t}\right)d\xi = 0.
\end{equation*}
Using the internal force balance \eqref{eq:internal_balance},
the above equation becomes 
\begin{equation}\label{eq:macro_balance1}
    \frac{\delta E}{\delta\phi} + \int_0^\infty \kappa(\xi)\Psi(\mathbf{x},t,\xi)\,d\xi = 0,
\end{equation}
which contains only elastic variables. The friction has been absorbed into the elastic history through the internal force balance. 

\begin{remark}\label{rmk:force-balance}
Equations \eqref{eq:relaxation}--\eqref{eq:macro_balance1} reveal a transparent mechanical
picture. For the order
parameter $\phi$, the chemical potential $\mu=\frac{\delta E}{\delta\phi}$ as its driving force is balanced by the total friction
  from all modes (as shown in \eqref{eq:macro_balance1}). For the
internal elastic network $\Psi(\cdot,\cdot, \xi)$ at each mode, its elastic restoring force is also balanced by friction, giving
  the Maxwell relaxation equation \eqref{eq:relaxation}.
  The sole coupling between them is the permeation friction
$\frac{\kappa(\xi)}{\xi}(\dot\phi - \dot\Psi)$, which acts as an
action--reaction pair.
For fast modes ($\xi$
large), $\Psi \approx 0$ but $\xi\Psi$ remains finite: their friction is
instantaneous and memoryless. For slow modes ($\xi$
small), $\Psi$ accumulates and persists, producing a long-lived friction that
retards $\phi$ beyond its Markovian dynamics. For broad spectra that place significant weight on small $\xi$ (e.g., the power-law $\kappa(\xi) \propto \xi^{\alpha-1}$ with small $\alpha$), the persistent friction dominates and further retards the evolution.

\end{remark}
\subsection{Energy dissipation law}\label{subsec:EDL}

We now establish the energy dissipation property, confirming that the model is thermodynamically consistent. This is an \emph{automatic} consequence of the Onsager variational structure.

Substituting the Euler--Lagrange equations
\eqref{eq:relaxation} and \eqref{eq:macro_balance1}
into  \eqref{eq:dFdt}, we obtain
\begin{equation}\label{eq:EDL_continuous}
    \begin{aligned}
        \frac{d\mathcal{F}}{dt}
    &= \int_\Omega\int_0^\infty \kappa(\xi)\,\Psi\left(\frac{\partial\Psi}{\partial t} - \frac{\partial\phi}{\partial t}\right)d\xi\,d\mathbf{x} \\
    &= -\int_\Omega\int_0^\infty \kappa(\xi)\,\xi\,\Psi^2\,d\xi\,d\mathbf{x} 
    \leq 0. 
    \end{aligned}
\end{equation}
Furthermore, comparing with the dissipation potential \eqref{eq:dissipation}:
\begin{equation}\label{eq:Onsager_balance}
    \frac{d\mathcal{F}}{dt} = -2\Phi \leq 0,
\end{equation}
recovering the classical Onsager energy--dissipation balance. The energy dissipation rate is entirely determined by the internal elastic strain --- a manifestation of the fact that all dissipation occurs through the dashpots in the Maxwell elements.

\begin{remark}
The energy dissipation law \eqref{eq:EDL_continuous}--\eqref{eq:Onsager_balance} holds for \emph{any} positive spectrum $\kappa(\xi)>0$, not only the special power-law choice shown in the next subsection. The \emph{thermodynamic consistency} is built into the Onsager framework and does not require a separate proof for each specific model.
\end{remark}

\subsection{The Caputo fractional derivative as a power-law special case}\label{subsec:recovery}

Substituting the solution \eqref{eq:Psi_solution_preview} into \eqref{eq:macro_balance1} yields:
\begin{equation}\label{eq:memory_integral}
    -\frac{\delta E}{\delta\phi} = \int_0^\infty \kappa(\xi)\left[\int_0^t e^{-\xi(t-s)}\frac{\partial\phi}{\partial s}\,ds\right]d\xi
    = \int_0^t \left[\int_0^\infty \kappa(\xi)\,e^{-\xi(t-s)}\,d\xi\right]\frac{\partial\phi}{\partial s}\,ds.
\end{equation}
Now we choose the power-law relaxation spectrum \eqref{eq:power_law_spectrum} with $C_\alpha = [\Gamma(\alpha)\Gamma(1-\alpha)]^{-1}$. Using the identity $\int_0^\infty \xi^{\alpha-1}e^{-\xi t}\,d\xi = \Gamma(\alpha)\,t^{-\alpha}$, equation \eqref{eq:memory_integral} becomes:
\begin{equation}\label{eq:Caputo_recovery}
    -\frac{\delta E}{\delta\phi}
    = \frac{1}{\Gamma(1-\alpha)}\int_0^t (t-s)^{-\alpha}\frac{\partial\phi}{\partial s}\,ds
    = \frac{\partial^\alpha\phi}{\partial t^\alpha},
\end{equation}
where $\partial^\alpha/\partial t^\alpha$ denotes the Caputo fractional derivative of order $\alpha\in(0,1)$.
Choosing $E(\phi) = \int_{\Omega}\left(\frac{\varepsilon^2}{2}|\nabla\phi|^2 + F(\phi)\right)d\mathbf{x}$ yields the Allen--Cahn type time-fractional phase-field equations (TFPFs):
\begin{equation}\label{fractional gradient flow}
\begin{cases}
    \displaystyle\frac{\partial^\alpha\phi}{\partial t^\alpha}(\mathbf{x},t) = -\mathcal{G}\mu(\mathbf{x},t), & (\mathbf{x},t)\in\Omega\times(0,T], \\[6pt]
    \mu(\mathbf{x},t) = -\varepsilon^2\Delta\phi(\mathbf{x},t) + F'(\phi(\mathbf{x},t)), \\[4pt]
    \phi(\mathbf{x},0) = \phi_0(\mathbf{x}), & \mathbf{x}\in\Omega,
\end{cases}
\end{equation}
subject to suitable (e.g., periodic or homogeneous Neumann) boundary conditions.

\begin{remark}[Conditions for recovering a standard fractional derivative]\label{rmk:conditions}
The Caputo derivative emerges as a special case requiring three conditions simultaneously:
\begin{enumerate}
    \item \textbf{Linear elasticity:} $W(\Psi) = \frac{1}{2}\kappa(\xi)|\Psi|^2$ (quadratic internal energy);
    \item \textbf{Linear coupling:} dissipation quadratic in $(\frac{\partial\phi}{\partial t} - \frac{\partial\Psi}{\partial t})$;
    \item \textbf{Power-law spectrum:} $\kappa(\xi) \propto \xi^{\alpha-1}$.
\end{enumerate}
Breaking any one of these leads to non-standard memory equations: nonlinear $W(\Psi)$ yields nonlinear memory (analogous to Finitely Extensible Nonlinear Elastic (FENE) vs.\ Hookean dumbbells in polymer kinetic theory); non-quadratic dissipation gives non-Newtonian internal friction; non-power-law $\kappa(\xi)$ produces non-fractional memory kernels (e.g., stretched exponential). This delineates the regime of validity of fractional phase-field models.
\end{remark}

\begin{remark}[Models beyond the fractional reduction]\label{rmk:beyond_fractional}
The same construction also gives non-fractional memory models. For example, a finite Prony-series spectrum yields a finite-dimensional integer-order system coupled to the phase field; a tempered or truncated spectrum produces algebraic memory with a finite terminal time scale; a mixture of several power-law spectra gives multi-term or distributed-order memory; and nonlinear internal energies or non-quadratic dissipation potentials lead to nonlinear hereditary phase-field equations. 
All of the above examples share a common mathematical structure: the memory kernel $K(t) = \int_0^\infty \kappa(\xi)\,e^{-\xi t}\,d\xi$
is completely monotone, as guaranteed by the positivity $\kappa(\xi)\geq 0$
(Bernstein's theorem \cite{Widder1941, Schilling2012}). This is a consequence of
thermodynamic consistency: a positive relaxation spectrum is precisely the
condition for nonnegative elastic energy and dissipation (cf. \eqref{eq:free_energy} and \eqref{eq:dissipation}). The
Bernstein representation also allows decomposition of such kernels into independent exponential modes,
which enables the efficient integer-order lifting developed in our numerical methods.
\end{remark}


\begin{remark}\label{rmk:CH}
The OVP derivation above yields the Allen--Cahn type equation
$\partial^\alpha\phi/\partial t^\alpha = -\mu$ because the dissipation is measured
in the $L^2$ norm (local friction on $\dot\phi$). To obtain the
Cahn--Hilliard type equation $\partial^\alpha\phi/\partial t^\alpha = \Delta\mu$
from first principles requires a different internal-variable construction in
which the Maxwell memory acts on the \emph{transport flux} rather than on the
order parameter itself. The physical picture concerns the interaction between elastic network and mass transport and is more complex. The corresponding OVP
derivation will be our future work.
Mathematically,  one can replace the dissipation operator $\mathcal{G} = I$ by
$\mathcal{G} = -\Delta$ in  \eqref{fractional gradient flow} to produce the time-fractional Cahn--Hilliard equation.
The resulting equation inherits the memory structure and admits the same numerical 
discretization method presented in Section~\ref{sec:2}. 
\end{remark}

\begin{remark}[Energy law for the power-law reduction]
For the power-law spectrum $\kappa(\xi)=C_\alpha\xi^{\alpha-1}$, the augmented energy-dissipation law \eqref{eq:EDL_continuous}--\eqref{eq:Onsager_balance} remains the primary thermodynamic structure. Eliminating the internal modes gives the Caputo phase-field equation, but the stored memory energy is then hidden in the history of the solution rather than in the classical phase-field energy alone. This explains why time-fractional phase-field equations generally require modified, integral, or nonlocal energy laws when they are written solely in terms of the reduced variable $\phi$.
\end{remark}

\section{Internal-variable lifting of MAPF equations for power-law relaxation spectra}
\label{sec:2}

The MAPF equations introduced in the previous section involve a continuous distribution of relaxation modes. 
Directly solving the continuous relaxation system
\eqref{eq:relaxation}--\eqref{eq:macro_balance1} is computationally challenging: the unbounded domain $\xi \in (0,\infty)$ with singular power-law weight, the extreme stiffness from multi-scale relaxation times, and the full-rank coupling in \eqref{eq:macro_balance1} all hinder efficient computation. For the power-law spectrum
$\kappa(\xi)=C_\alpha \xi^{\alpha-1},$
the relaxation variables can be formally eliminated, leading to the corresponding fractional phase-field model \eqref{fractional gradient flow}. However, the resulting fractional formulation contains nonlocal temporal operators, and straightforward time discretizations usually require storing and repeatedly evaluating the entire solution history. Nevertheless, the formulation in Section \ref{sec:derivation} reveals a key insight that memory effects can be converted into local-in-time integer-order dynamics through internal-variable lifting.

Guided by this insight, we introduce in this section a transformed family of auxiliary variables that maps the relaxation spectrum onto the bounded interval $(0,1)$ with a Jacobi-type weight, yielding an equivalent local-in-time formulation of the MAPF equations amenable to spectral discretization and fast implementation. 
The resulting formulation preserves the thermodynamic properties of the original system and provides a convenient framework for efficient numerical discretization.

\subsection{Internal-variable transformation for the power-law memory kernel}
\label{subsec:Onsager_EPPF}

In this subsection, we construct an equivalent internal-variable formulation of the MAPF equations for the power-law relaxation spectrum. Starting from the relaxation representation of the memory kernel, we introduce a transformed family of auxiliary variables and derive a local-in-time formulation that avoids explicit history convolutions while preserving the underlying memory effects.

In view of \eqref{eq:Caputo_recovery}, we apply the fractional integral operator $I_t^{1-\alpha}$ to
the relaxation equation \eqref{eq:relaxation} to obtain
\[
I_t^{1-\alpha}
\left(
\frac{\partial\Psi}{\partial t}
+\xi\Psi
\right)
=
I_t^{1-\alpha}
\frac{\partial\phi}{\partial t},
\]
where $I_t^{\alpha}\phi(\mathbf{x},t) = \frac{1}{\Gamma(\alpha)}\int_0^t(t-s)^{\alpha-1}\phi(\mathbf{x},s)ds$. Under the undeformed reference state condition, i.e., $\Psi(\mathbf{x},0,\xi)=0$, the above equation can be rewritten as
\begin{equation}\label{equation of Psi}
\frac{\partial}{\partial t}
I_t^{1-\alpha}\Psi(\mathbf{x},t,\xi)
+\xi I_t^{1-\alpha}\Psi(\mathbf{x},t,\xi)
=
I_t^{1-\alpha}
\frac{\partial\phi}{\partial t}(\mathbf{x},t),
\quad \xi\in(0,+\infty).
\end{equation}

Next, we map the relaxation-rate variable $\xi\in(0,\infty)$ to an extended variable $\theta\in(0,1)$ by 
\[
\xi=c_1(\theta)=\frac{\theta}{1-\theta},
\]
and introduce 
\begin{equation}\label{relationship between varphi and Psi}
\varphi(\mathbf{x},t,\theta)
=
\phi_0(\mathbf{x})
+
c_0(\theta)
I_t^{1-\alpha}
\Psi(\mathbf{x},t,c_1(\theta)),
\end{equation}
where $c_0(\theta) = \frac{1}{1-\theta}$. It follows from the above equation and \eqref{equation of Psi} that
\begin{equation*}\label{EPPFs without fractional}
\frac{\partial\varphi}{\partial t}
(\mathbf{x},t,\theta)
+
c_1(\theta)
\left(
\varphi(\mathbf{x},t,\theta)-\phi_0(\mathbf{x})
\right)
= 
c_0(\theta)
I_t^{1-\alpha}
\frac{\partial\phi}{\partial t}
(\mathbf{x},t),
\quad \theta\in(0,1).
\end{equation*}
Consequently, substituting  \eqref{eq:Caputo_recovery} into the above equation yields
\begin{equation}\label{EPDEs0}
\begin{aligned}
&
\frac{\partial\varphi}{\partial t}
(\mathbf{x},t,\theta)
+
c_1(\theta)
(\varphi(\mathbf{x},t,\theta)-\phi_0(\mathbf{x}))
=
-c_0(\theta)
\frac{\delta E}{\delta\phi}(\mathbf{x},t),
\\
&\varphi(\mathbf{x},0,\theta)=\phi_0(\mathbf{x}).
\end{aligned}
\end{equation}

Next, we establish an explicit relation between $\varphi(\mathbf{x},t,\theta)$ and $\phi(\mathbf{x},t)$. Taking the
Laplace transform of the relaxation equation \eqref{eq:relaxation} gives
\begin{equation*}\label{sec3:01}
\widehat{\Psi}(\mathbf{x},s,c_1(\theta))
=
\frac{s\widehat{\phi}(\mathbf{x},s)-\phi_0(\mathbf{x})}
{s+c_1(\theta)}.
\end{equation*}
Meanwhile, taking the Laplace transform of
\eqref{relationship between varphi and Psi}, we have
\begin{equation*}\label{sec3:02}
\widehat{\varphi}(\mathbf{x},s,\theta)
-\frac1s\phi_0(\mathbf{x})
=
c_0(\theta)s^{\alpha-1}
\widehat{\Psi}
(\mathbf{x},s,c_1(\theta)).
\end{equation*}
It follows from the above two equations that 
\begin{equation}\label{eq:phihat}
    \widehat{\varphi}(\mathbf{x},s,\theta)
-\frac1s\phi_0(\mathbf{x})
=
c_0(\theta)s^{\alpha-1}
\frac{s\widehat{\phi}(\mathbf{x}, s)-\phi_0(\mathbf{x})}
{s+c_1(\theta)}.
\end{equation}
Let 
\begin{equation}\label{omage}
    \omega_\alpha(\theta) = \frac{\theta^{-\alpha}(1-\theta)^{\alpha-1}}{\Gamma(1-\alpha)\Gamma(\alpha)}.
\end{equation}
Then, we multiply \eqref{eq:phihat} by
$\omega_\alpha(\theta)$ and integrate over $(0,1)$ to get 
\begin{equation}\label{Laplace for varphi}
\int_0^1
\widehat{\varphi}
(\mathbf{x},s,\theta)
\omega_\alpha(\theta)d\theta
=
\frac{A}{s}\phi_0(\mathbf{x})
+
s^{\alpha-1}
(s\widehat{\phi}(\mathbf{x},s)-\phi_0(\mathbf{x}))
B(s),
\end{equation}
where
\begin{equation*}
A=\int_0^1\omega_\alpha(\theta)d\theta = 1,\quad B(s)=
\int_0^1
\frac{c_0(\theta)}
{s+c_1(\theta)}
\omega_\alpha(\theta)d\theta 
= s^{-\alpha}.
\end{equation*}
Here we calculate $B(s)$ by using the Euler-Gauss hypergeometric
integral formula \cite{Abramowitz1965}.
Applying the inverse Laplace transform to
\eqref{Laplace for varphi}, we recover
\begin{equation}\label{eq:transform}
\phi(\mathbf{x},t)
=
\int_0^1
\varphi(\mathbf{x},t,\theta)
\omega_\alpha(\theta)d\theta
:=
\mathcal{C}[\varphi](\mathbf{x},t).
\end{equation}

Consequently, we obtain from \eqref{EPDEs0} the following \emph{internal-variable formulation of the MAPF equations}:
\begin{equation}\label{EPDEs}
    \begin{aligned}
        &\frac{\partial \varphi}{\partial t}(\mathbf{x},t,\theta) + c_{1}(\theta)\left(\varphi(\mathbf{x},t,\theta) - \phi_{0}(\mathbf{x})\right) = -c_{0}(\theta)\mathcal{G}\mu(\mathbf{x},t), \\
        & \mu(\mathbf{x},t) = -\varepsilon^{2}\Delta(\mathcal{C}[\varphi](\mathbf{x},t)) + f(\mathcal{C}[\varphi](\mathbf{x},t)),\quad \mathbf{x}\in\Omega, \theta\in I = (0,1),t\in(0,T],\\
        &\varphi(\mathbf{x},0,\theta) = \phi_{0}(\mathbf{x}),\quad \mathbf{x}\in\Omega, \theta\in I = (0,1).
    \end{aligned}
\end{equation}
The original phase-field variable $\phi$ is recovered through the averaging operator \eqref{eq:transform}.


\begin{remark}
 
We note that the internal-variable formulation \eqref{EPDEs} is mathematically equivalent, in a weak sense, to the augmented integer-order formulation (3.11) in \cite{Fritz2021}, where the well-posedness of the corresponding problem has been established. However, the two formulations arise from fundamentally different modeling perspectives. In \cite{Fritz2021}, the augmented system is constructed from a prescribed time-fractional gradient flow to reveal its hidden integer-order structure and associated energy structure. In contrast, the present formulation is derived from a thermodynamically consistent relaxation-spectrum-based model, where the internal variables represent physical relaxation modes associated with stored memory energy and dissipation mechanisms. Therefore, although the resulting systems are equivalent in the power-law case, the present derivation provides a different physical interpretation of memory effects and naturally extends to more general relaxation spectra beyond the fractional setting.
\end{remark}

\subsection{Energy dissipation and mass conservation}
The above internal-variable formulation preserves the essential thermodynamic properties of the original MAPF equations. In particular, it satisfies the energy dissipation law and, in the Cahn-Hilliard case with $\mathcal{G}=-\Delta$, conserves the total mass of the phase variable.

Before presenting the energy dissipation law, we define the following weights:
\begin{equation}\label{OME}
    \omega_{\alpha,0}(\theta) = (1-\theta)\omega_{\alpha}(\theta),\quad \omega_{\alpha,1}(\theta) = \theta\omega_{\alpha}(\theta),
\end{equation}
and the weighted inner product and norm:
\begin{equation*}
\begin{aligned}
    &(u,w)_{L^{2}(\Omega;L_{\omega_{\alpha}}^{2}(I))} = \int_{\Omega}\int_{0}^{1}uw\omega_{\alpha}(\theta)d\theta d\mathbf{x},\\
    & \Vert u\Vert^{2}_{L^{2}(\Omega;L_{\omega_{\alpha}}^{2}(I))} = \int_{0}^{1}\Vert u\Vert_{L^{2}(\Omega)}^{2}\omega_{\alpha}(\theta)d\theta,
\end{aligned}
\end{equation*}
where $\omega_{\alpha}(\theta)$ is given in \eqref{omage}. 

The energy dissipation law for the internal-variable formulation \eqref{EPDEs} is given as follows:
\begin{theorem}\label{energy dissipation for AC}
    Let $\varphi$ be the solution of \eqref{EPDEs}. Then the following energy dissipation law holds:
    \begin{equation}
        \label{eq:energy_law_metric}
        \frac{d\mathcal{E}(\varphi)}{dt} = -\left\Vert\mathcal{G}^{-1/2}\frac{\partial \varphi}{\partial t}\right\Vert^2_{L^{2}(\Omega;L_{\omega_{\alpha,0}}^{2}(I))},\quad t>0,
    \end{equation}
    where
    \begin{equation}\label{energy of AC}
        \mathcal{E}(\varphi) = E(\mathcal{C}[\varphi])  + \frac{1}{2}\left\Vert\mathcal{G}^{-1/2}(\varphi - \phi_{0})\right\Vert^2_{L^2(\Omega;L_{\omega_{\alpha,1}}^{2}(I))}.
    \end{equation} 
    
\end{theorem}
\begin{proof}
    Testing \eqref{EPDEs} with $\omega_{\alpha,0}(\theta)\,\mathcal{G}^{-1}\frac{\partial \varphi}{\partial t}$ gives 
    \begin{equation}\label{energy_dissipation_1}
        \begin{aligned}
            &\left(\mathcal{G}^{-1}\frac{\partial \varphi}{\partial t},\frac{\partial \varphi}{\partial t}\right)_{L^{2}(\Omega;L_{\omega_{\alpha,0}}^{2}(I))} + \left(\mathcal{G}^{-1}(\varphi - \phi_{0}),\frac{\partial \varphi}{\partial t}\right)_{L^{2}(\Omega;L_{\omega_{\alpha,1}}^{2}(I))}\\
            & = \varepsilon^{2}\left(\Delta (\mathcal{C}[\varphi]),\frac{\partial \varphi}{\partial t}\right)_{L^{2}(\Omega;L_{\omega_{\alpha}}^{2}(I))} - \left(f(\mathcal{C}[\varphi]),\frac{\partial \varphi}{\partial t}\right)_{L^{2}(\Omega;L_{\omega_{\alpha}}^{2}(I))},
        \end{aligned}
    \end{equation}
    where $\omega_{\alpha,1}(\theta), \omega_{\alpha,0}(\theta)$ are given in \eqref{OME}. Note that $\phi = \mathcal{C}[\varphi] = \int_{0}^{1}\varphi(\mathbf{x},t,\theta)\omega_{\alpha}(\theta)d\theta$, thus we have
    \begin{equation*}
    \begin{aligned}
         \varepsilon^{2}\left(\Delta (\mathcal{C}[\varphi]),\frac{\partial \varphi}{\partial t}\right)_{L^{2}(\Omega;L_{\omega_{\alpha}}^{2}(I))} &= -\varepsilon^{2}\left(\nabla(\mathcal{C}[\varphi]),\nabla(\int_{0}^{1}\frac{\partial \varphi}{\partial t}(\mathbf{x},t,\theta)\omega_{\alpha}(\theta)d\theta)\right)_{L^{2}(\Omega)}\\
         & = -\frac{d}{dt}\left(\frac{\varepsilon^{2}}{2}\Vert\nabla(\mathcal{C}[\varphi])\Vert^2_{L^{2}(\Omega)}\right).
    \end{aligned}
    \end{equation*}
     Applying the same argument used in the above equation to the last term of \eqref{energy_dissipation_1}, we conclude that 
    \begin{equation*}
    \begin{aligned}
        &\frac{d}{dt}\int_{\Omega}\left(\frac{\varepsilon^{2}}{2}\vert\nabla(\mathcal{C}[\varphi])\vert^{2} + F(\mathcal{C}[\varphi]) + \frac{1}{2}\left(\mathcal{G}^{-1}(\varphi - \phi_{0}),\varphi - \phi_{0}\right)_{L_{\omega_{\alpha,1}}^{2}(I)}\right)dx \\
        & = -\left\Vert\mathcal{G}^{-1/2}\frac{\partial \varphi}{\partial t}\right\Vert^2_{L^{2}(\Omega;L_{\omega_{\alpha,0}}^{2}(I))},
    \end{aligned}   
    \end{equation*}
    where we use the equality $\frac{\partial }{\partial t} (\varphi - \phi_{0}) = \frac{\partial \varphi }{\partial t}$, and $\phi_{0}$ is given in \eqref{fractional gradient flow}.
\end{proof}
\begin{remark}
The energy $\mathcal{E}(\varphi)$ defined in \eqref{energy of AC} has the same structural form as the free energy $\mathcal{F}(\phi,\Psi)$ in \eqref{eq:free_energy}: both consist of a Ginzburg--Landau energy term plus a strictly positive term representing stored memory. Compared with the variational energy functional $\mathcal{E}_{\alpha}$ proposed in \cite{LiaoTao2021} and the nonlocal energy $\overline{E}$ proposed in \cite{HouXu2021}, the energy \eqref{energy of AC} has a direct internal-variable interpretation.

\end{remark}

We now present the mass-conservation properties for the Cahn--Hilliard case with $\mathcal{G} = -\Delta$.
\begin{theorem}\label{mass_preserving}
    Let $\varphi$ be the solution of \eqref{EPDEs} with $\mathcal{G} = -\Delta$. Then $\varphi$ preserves the following global mass-conservation and point-wise (in $\theta$) mass-conservation: 
    \begin{align}
        &\int_{\Omega}\mathcal{C}[\varphi](\mathbf{x},t)d\mathbf{x} = \int_{\Omega}\phi_0(\mathbf{x})d\mathbf{x},\quad t\geq 0,\label{global mass conservation} \\ 
&\int_{\Omega}\varphi(\mathbf{x},t,\theta)d\mathbf{x} = \int_{\Omega}\phi_0(\mathbf{x})d\mathbf{x},\quad t\geq0,\quad \theta\in (0,1).\label{pointwise mass conservation}
    \end{align}
\end{theorem}
\begin{proof}
    Testing \eqref{EPDEs} with $\omega_{\alpha}(\theta)\, v$ gives 
    \begin{equation*}
    \begin{aligned}
        &\left(\frac{\partial \varphi(t)}{\partial t},v\right)_{L^{2}(\Omega;L^{2}_{\omega_{\alpha}}(I))} + \left(c_1(\theta)(\varphi - \phi_0),v\right)_{L^{2}(\Omega;L^{2}_{\omega_{\alpha}}(I))} \\
        & =  -\left(\varepsilon^{2}c_{0}(\theta)(-\Delta)(\mathcal{C}[\varphi]) + c_{0}(\theta)f(\mathcal{C}[\varphi]),(-\Delta)v\right)_{L^{2}(\Omega;L^{2}_{\omega_{\alpha}}(I))}.
    \end{aligned}
 \end{equation*}
 Taking $v = 1$ in the above equation, since $(-\Delta)v = 0$, we obtain 
 \begin{equation}\label{mass_conti:01}
     \left(\frac{\partial \varphi(t)}{\partial t},1\right)_{L^{2}(\Omega;L^{2}_{\omega_{\alpha}}(I))} =- \left(c_1(\theta)(\varphi - \phi_0),1\right)_{L^{2}(\Omega;L^{2}_{\omega_{\alpha}}(I))}.
 \end{equation}
Next, taking the $L^2(\Omega)$ inner product of \eqref{EPDEs} with $v = 1$ yields 
\begin{equation}\label{mass_conti:02}
    \frac{\partial}{\partial t}\int_\Omega \left(\varphi-\phi_0\right) d\mathbf{x} + c_1(\theta)\int_\Omega\left(\varphi - \phi_0 \right)d\mathbf{x} = 0.
\end{equation}
Define $\chi(t,\theta) = \int_\Omega\left(\varphi - \phi_0\right)d\mathbf{x}$. Then the above equation can be reformulated as the following parametric ODEs:
\begin{equation*}
    \chi' + \sigma \chi = 0, \quad \chi(0,\theta) = 0,
\end{equation*}
with $\sigma = c_1(\theta)> 0$. The exact solution of the above ODE is given by:
\begin{equation}\label{mass_conti:03}
    \chi = \int_\Omega\left(\varphi - \phi_0 \right)d\mathbf{x} = 0.
\end{equation}
Therefore, integrating \eqref{mass_conti:01} and \eqref{mass_conti:02} with respect to $t$ from $0$ to $t$ gives
\begin{equation*}
\begin{aligned}
    \int_{\Omega}\mathcal{C}[\varphi](\mathbf{x},t)d\mathbf{x} = \int_{\Omega}\phi_0(\mathbf{x})d\mathbf{x},\quad \int_{\Omega}\varphi(\mathbf{x},t,\theta)d\mathbf{x} = \int_{\Omega}\phi_0(\mathbf{x})d\mathbf{x},\quad t\geq0,\quad \theta\in (0,1),
\end{aligned}
\end{equation*}
thanks to \eqref{mass_conti:03}.
\end{proof}
\begin{remark}

According to Theorem \ref{mass_preserving}, the formulation \eqref{EPDEs} preserves both global mass conservation and pointwise mass conservation with respect to the relaxation variable $\theta$. The latter represents a stronger property, indicating that the solution associated with each individual relaxation mode preserves the total mass. 
This property provides a theoretical foundation for developing high-fidelity numerical schemes for memory-dependent phase-field dynamics.
\end{remark}
\begin{remark}
    For the time-fractional formulation, the interchange between the spatial integral and the Caputo derivative relies on the definition of the Caputo derivative and Fubini's theorem \cite{Kilbas2006Theory}. In particular, under the regularity assumption
\[
\phi\in C^{0,\alpha}(0,T;L^1(\Omega)),
\]
one has
\begin{equation*}
    \int_{\Omega}\frac{\partial^\alpha\phi}{\partial t^\alpha}d\mathbf{x}
    =
    \frac{d^\alpha}{dt^\alpha}
    \int_{\Omega}\phi d\mathbf{x}.
\end{equation*}

In contrast, for the internal-variable formulation, the corresponding relation involves only the classical weak time derivative of the averaged variable \(\mathcal C[\phi]\). For instance, if
\[
\phi = \mathcal C[\varphi]\in W^{1,1}(0,T;L^1(\Omega)),
\]
then
\[
    \int_{\Omega}\frac{\partial\mathcal{C}[\varphi]}{\partial t}d\mathbf{x}
    =
    \frac{d}{dt}\int_{\Omega}\mathcal{C}[\varphi]d\mathbf{x}
\]
holds in the distributional sense. Thus, the internal-variable formulation avoids the more delicate interchange between a spatial integral and a nonlocal Caputo derivative, and provides a more direct framework for establishing mass conservation.
\end{remark}

\section{ESAV fully discrete schemes and fast implementation}\label{sec:3}
In this section, we propose  first- and second-order stabilized ESAV schemes~(cf. \cite{sESAV_2022}) for the time discretization, while we use the Fourier-spectral method in space and the Jacobi-spectral collocation method for the internal-variable. We also propose a fast rank-one algorithm for the proposed fully discrete schemes. Under periodic boundary conditions, we establish the discrete structure-preserving properties of the proposed schemes, including discrete energy dissipation and mass conservation.

The numerical methods are developed for the internal-variable formulation of MAPF equations rather than the fractional representation, since the former avoids explicit history evaluation and provides a local-in-time computational framework.

\subsection{Fully discrete scheme}

Let us first introduce Fourier spectral method and Jacobi spectral collocation method. We partition the domain $\Omega = [0,2\pi)^{2}$ uniformly with size $h_{l} = 2\pi/N_{l},\,l= x,y$, where $N_{x}, N_{y}$ are the positive integers. Let $\{\theta_s\}_{s=0}^M$ be the Jacobi-Gauss points in $(0,1)$ with respect to the weight $\Gamma(\alpha) \Gamma(1-\alpha)\omega_{\alpha}(\theta) = \theta^{-\alpha}(1-\theta)^{\alpha - 1}$, and $\{h_s(\theta)\}_{s=0}^M$ be the Lagrange interpolation functions with respect to the Jacobi-Gauss points $\{\theta_s\}_{s=0}^M$, respectively. 

In what follows, we take $N = N_{x} = N_{y}$ for simplicity. The fully discrete approximation space is:
\begin{equation*}
    \mathcal{M}(\Omega\times I) = span\left\{e^{ik x_{1}}e^{im x_{2}}h_{j}(\theta):-\frac{N}{2}\leq k,m\leq \frac{N}{2}- 1,0\leq j\leq M\right\},
\end{equation*}
where $i = \sqrt{-1}$.

 Then any function $\varphi(x_{1},x_{2},\theta)\in \Omega\times I$ can be approximated by :
\begin{equation*}
    \varphi(x_{1},x_{2},\theta)\approx \varphi_{NM}(x_{1},x_{2},\theta) = \sum\limits_{j=0}^{M}\sum\limits_{k=-N/2}^{N/2-1}\sum\limits_{m=-N/2}^{N/2-1}\widehat{\varphi}_{km}(x_1,x_2,\theta_{j})e^{ikx_{1}}e^{imx_{2}}h_{j}(\theta),
\end{equation*}
where the Fourier coefficients are given by 
\begin{equation*}\label{fourier coefficient}
    \widehat{\varphi}^{n+1}_{km}(x_1,x_2,\theta_{j}) = \frac{1}{\vert\Omega\vert}\int_{\Omega}\varphi^{n+1}(x_{1},x_{2},\theta_{j})e^{-i(kx_{1} + mx_{2})}d\mathbf{x},
\end{equation*}
and the function $\phi_{NM}:=\mathcal{C}[\varphi_{NM}]$ is computed by using the Gauss-Jacobi quadrature, namely, 
\begin{equation*}\label{Cphi}
\begin{aligned}
     \mathcal{C}[\varphi_{NM}](x_1,x_2) =  \int_{0}^{1} \varphi_{NM}(x_1,x_2,\theta)\omega_{\alpha}(\theta)d\theta
    =   \sum\limits_{j = 0}^{M}\varphi_{NM}(x_1,x_2,\theta_{j})\omega_{j},
\end{aligned}
\end{equation*}
where $\omega_{j} =  \int_{0}^{1}h_{j}(\theta)\omega_{\alpha}(\theta)d\theta,\, j = 0,\cdots,M$ are the corresponding Jacobi-Gauss weights.

To construct the stabilized ESAV schemes for \eqref{EPDEs}, we introduce  $s(t) = E_1(\mathcal{C}[\varphi]) = \int_{\Omega}F(\mathcal{C}[\varphi])d\mathbf{x}$, then we rewrite \eqref{EPDEs} as the following system:
 \begin{equation}\label{sESAV}
     \begin{aligned}
         &(1-\theta)\frac{\partial\varphi}{\partial t} + \theta \left(\varphi-\phi_0\right) = -\mathcal{G}\mu,\\
         &\mu = -\varepsilon^2\Delta\left(\mathcal{C}[\varphi]\right) + g\left(\mathcal{C}[\varphi],s(t)\right)f\left(\mathcal{C}[\varphi]\right),\\
         & \frac{d s(t)}{dt} = g\left(\mathcal{C}[\varphi],s(t)\right)\int_{\Omega}f\left(\mathcal{C}[\varphi]\right)\frac{\partial\mathcal{C}[\varphi]}{\partial t}d\mathbf{x},
     \end{aligned}
 \end{equation}
 where $g\left(\mathcal{C}[\varphi],s(t)\right) = \exp(s(t))/\exp(E_1(\mathcal{C}[\varphi]))$. Let $N_T>0$ be a positive integer and set 
\begin{equation*}
    \Delta t = T/N_{T},\quad t_{n} = n\Delta t\quad \text{for} \quad n\leq N_{T}.
\end{equation*}
%

We then present the first- and second-order schemes for \eqref{sESAV} as follows:
\begin{itemize}
    \item \textbf{Scheme I (First order scheme)}: Find $\varphi^{n+1}_{NM}\in \mathcal{M}(\Omega\times I)$, such that:
\begin{equation}\label{ESAV_first_order_fully}
    \begin{aligned}
    &\frac{1}{\Delta t}\left(\mathcal{G}^{-1}(\varphi_{NM}^{n+1} - \varphi_{NM}^{n}),q_{NM}\right)_{L^2(\Omega;L^2_{\omega_{\alpha,0}}(I))}\\
    &+ \left(\mathcal{G}^{-1}(\varphi_{NM}^{n+1} - \phi_{0}),q_{NM}\right)_{L^2(\Omega;L^2_{\omega_{\alpha,1}}(I))}\\
    &=- \varepsilon^{2}\left( \nabla \left(\mathcal{C}[\varphi_{NM}^{n+1}]\right), \nabla \left(\mathcal{C}[q_{NM}]\right)\right)_{L^2(\Omega)}  \\
    &-\left(g(\mathcal{C}[\varphi_{NM}^n],s^n)f(\mathcal{C}[\varphi^{n}_{NM}]),\mathcal{C}[q_{NM}]\right)_{L^2(\Omega)}\\
    & - S\left(\mathcal{G}(\varphi_{NM}^{n+1}-\varphi_{NM}^{n}),q_{NM}\right)_{L^2(\Omega;L^2_{\omega_{\alpha}}(I))},\\
    & s^{n+1} - s^n = g(\mathcal{C}[\varphi_{NM}^{n}],s^n)\int_{\Omega}f(\mathcal{C}[\varphi_{NM}^{n}])\mathcal{C}[\varphi_{NM}^{n+1} - \varphi_{NM}^{n}]d\mathbf{x},
    \end{aligned}
\end{equation}
where $\forall q_{NM}\in \mathcal{M}(\Omega\times I)$ and $S\geq0$ is a stabilization constant.

\item \textbf{Scheme II (Second order scheme)}: Find $\varphi^{n+1}_{NM}\in \mathcal{M}(\Omega\times I)$, such that:
\begin{equation}\label{ESAV_second_order_fully}
    \begin{aligned}
       &\frac{1}{\Delta t}\left(\mathcal{G}^{-1}(\varphi_{NM}^{n+1} - \varphi_{NM}^{n}),q_{NM}\right)_{L^2(\Omega;L^2_{\omega_{\alpha,0}}(I))}\\
       &+ \frac{1}{2}\left(\mathcal{G}^{-1}(\varphi_{NM}^{n+1}+\varphi_{NM}^{n} - 2\phi_{0}),q_{NM}\right)_{L^2(\Omega;L^2_{\omega_{\alpha,1}}(I))} \\
       &=-\frac{\varepsilon^{2}}{2}\left( \nabla \left(\mathcal{C}[\varphi_{NM}^{n+1} + \varphi_{NM}^{n}]\right), \nabla \left(\mathcal{C}[q_{NM}]\right)\right)_{L^2(\Omega)}\\
       &- \left(g(\mathcal{C}[\varphi_{NM}^{n+\frac{1}{2}}],s^{n+\frac{1}{2}})f(\mathcal{C}[\varphi^{n+\frac{1}{2}}_{NM}]),\mathcal{C}[q_{NM}]\right)_{L^2(\Omega)}\\
       & - S\left(\mathcal{G}(\varphi_{NM}^{n+1}-2\varphi_{NM}^{n}+\varphi_{NM}^{n-1}),q_{NM}\right)_{L^2(\Omega;L^2_{\omega_{\alpha}}(I))},\\
    & s^{n+1} - s^n = g(\mathcal{C}[\varphi_{NM}^{n+\frac{1}{2}}],s^{n+\frac{1}{2}})\int_{\Omega}f(\mathcal{C}[\varphi_{NM}^{n+\frac{1}{2}}])\mathcal{C}[\varphi^{n+1}_{NM}-\varphi^{n}_{NM}]d\mathbf{x},
    \end{aligned}
\end{equation}
where $\varphi^{n+\frac{1}{2}} = \frac{3}{2}\varphi^{n} - \frac{1}{2}\varphi^{n-1}$, $\forall q_{NM}\in \mathcal{M}(\Omega\times I)$.
\end{itemize}

In the following, we only give the implementation details and theoretical results for the second-order scheme. For the first-order scheme, these can be easily shown in a similar way.

 \subsection{Fast rank-one algorithm}
 In our previous work \cite{DingMaoLiu}, we proposed a fast algorithm based on the generalized characteristic decomposition. However, the computational cost of the proposed fast algorithm is dominated by dense matrix-vector multiplications, requiring two such multiplications per time step.
 
 To achieve further acceleration, we develop the following fast rank-one algorithm that avoids characteristic decomposition and only involves sparse matrix multiplications. To this end, we rewrite the linear system of the fully discrete scheme \eqref{ESAV_second_order_fully}:
 \begin{equation*}
         \mathbf{A}\mathbf{\widehat{U}}^{n+1} + \mathbf{B}\mathbf{\widehat{U}}^{n+1}\mathbf{C} = \mathbf{D}\mathbf{\widehat{U}}^{n} - \mathbf{B}\mathbf{\widehat{U}}^{n}\mathbf{C} + \mathbf{F},
     \end{equation*}
      where 
    \begin{equation*}
        \begin{aligned}
            &\mathbf{A} = {\rm{diag}}(A) + S\Delta t  \mathbf{I}_{M},\quad A = \left[1-\theta_{0} + \frac{\Delta t}{2} \theta_{0},\cdots,1-\theta_{M} + \frac{\Delta t}{2} \theta_{M}\right]^{T},\\
            & \mathbf{B} = \mathbf{1}_M\rm{W}^T,\quad \mathbf{1}_M = [1,1,\cdots,1]^T ,\quad  \rm{W} = \left[\omega_{0},\omega_{1},\cdots,\omega_{M}\right]^{T},\\
            & \mathbf{C} = \frac{\varepsilon^2 \Delta t}{2}\left[ K_1^2 + (K_2^2)^T\right],K_1,K_2 = \left[0,1,\cdots,N/2-1,-N/2,\cdots,-1\right].\\
            &\mathbf{D} = {\rm{diag}}(D) + 2S\Delta t  \mathbf{I}_{M},\quad D = \left[1-\theta_{0} - \frac{\Delta t}{2}\theta_0,\cdots,1-\theta_{M} - \frac{\Delta t}{2}\theta_M\right]^T,\\
            & \mathbf{F} = \Delta t \hat{\phi}^{0}\otimes{\rm{diag}}\left([\theta_{0},\cdots,\theta_{M}]\right) - \Delta t\widehat{f}(\mathcal{C}[\varphi^{n+\frac{1}{2}}_{NM}])\otimes \mathbf{I}_{M} - S\Delta t\mathbf{\widehat{U}}^{n-1},\\
            & \widehat{\mathbf{U}}^{n+1} = \left[\widehat{\varphi}^{n+1}_{km}(\theta_{0}),\cdots,\widehat{\varphi}^{n+1}_{km}(\theta_{M})\right]^{T},-\frac{N}{2}\leq k,m\leq \frac{N}{2}-1,
        \end{aligned}
    \end{equation*}    
    where $\Delta t$ is time step size, $\mathbf{I}_{M}$ is the identity operator, and $\otimes$ is the tensor product.
    
It is worth noting that  $\mathbf{B}$ is a rank-one matrix whose rows are all identical, i.e.,
     \begin{equation*}
         \mathbf{B} = \mathbf{1}_M\rm{W}^T,\quad \mathbf{1}_M = [1,1,\cdots,1]^T \in \mathbb{R}^{(M+1)\times1},\quad  \rm{W}^T\in \mathbb{R}^{1\times(M+1)}.
     \end{equation*}

     Hence, we reformulate the matrix form as follows: for $\mathit{k,m} \in K,M,$
     \begin{equation}\label{fast_rank1:01}
\mathbf{A}\mathbf{\widehat{U}}_{km}^{n+1}+\mathbf{c}_{km}\mathbf{1}_{M}(\rm{W}^T\mathbf{\widehat{U}}_{\mathit{km}}^{\mathit{n}+1}) = \mathbf{D}\mathbf{\widehat{U}}_{\mathit{km}}^{\mathit{n}} - \mathbf{c}_{\mathit{km}}\mathbf{1}_{M}(\rm{W}^T\mathbf{\widehat{U}}_{\mathit{km}}^{\mathit{n}}) + \mathbf{F}_{\mathit{km}},
     \end{equation}
     where $\mathbf{F}_{km}\in \mathbb{R}^{(M+1)\times 1}$, $\mathbf{c}_{km}$ denotes the element in the 
$k$-th row and 
$m$-th column of the discrete Laplace operator matrix $\mathbf{C}$. 

Denote $\gamma_{km}^{n+1} = \rm{W}^T\mathbf{\widehat{U}}^{\mathit{n}+1}_{\mathit{km}}\in \mathbb{R}^{1\times 1}$, $\mathbf{H}^{n}_{km} = \mathbf{D}\mathbf{\widehat{U}}_{\mathit{km}}^{n} + \mathbf{F}_{km}$. Multiplying  both sides of \eqref{fast_rank1:01} by $\mathbf{A}^{-1}$, we obtain:
     \begin{equation}\label{fast_rank1:02}
         \mathbf{\widehat{U}}_{km}^{n+1} = \mathbf{A}^{-1}\mathbf{H}^{n}_{km}- \mathbf{c}_{km}\gamma^{n+1}_{km}\mathbf{A}^{-1}\mathbf{1}_{M}- \mathbf{c}_{km}\gamma_{km}^{n}\mathbf{A}^{-1}\mathbf{1}_{M}. 
     \end{equation}
     Note that $\mathbf{A}$ is a diagonal matrix, then $\mathbf{A}^{-1}\mathbf{H}^{n}_{km}, \mathbf{A}^{-1}\mathbf{1}_{M}$ are vectors in $\in \mathbb{R}^{(M+1)\times 1}$.

     Next, multiplying both sides of \eqref{fast_rank1:02} by $\rm{W}^T$ yields
     \begin{equation*}\label{fast_rank1:03}
         \gamma_{km}^{n+1} = \rm{W}^T\left(\mathbf{A}^{-1}\mathbf{H}^{\mathit{n}}_{\mathit{km}}\right) - \mathbf{c}_{\mathit{km}} \gamma^{\mathit{n}}_{\mathit{km}}\rm{W}^T(\mathbf{A}^{-1}\mathbf{1}_{M})- \mathbf{c}_{\mathit{km}} \gamma^{\mathit{n}+1}_{\mathit{km}}\rm{W}^T(\mathbf{A}^{-1}\mathbf{1}_{M}).
     \end{equation*}
     Then we have
     \begin{equation}\label{fast_rank1:04}
         \gamma^{\mathit{n}+1}_{km} = \frac{ - \mathbf{c}_{km}\beta}{1 + \mathbf{c}_{km}\beta}\gamma_{km}^n + \frac{\alpha_{km}}{1 + \mathbf{c}_{km}\beta},
     \end{equation}
     where 
     \begin{equation*}
         \alpha_{km} = \rm{W}^T\left(\mathbf{A}^{-1}\mathbf{H}^{\mathit{n}}_{\mathit{km}}\right),\quad \beta = \rm{W}^T(\mathbf{A}^{-1}\mathbf{1}_{M}).
     \end{equation*}
     Once we solve $\gamma_{km}^{n+1}$ by the above equation, the solution  $\mathbf{\widehat{U}}_{km}^{n+1}$ immediately follows from  \eqref{fast_rank1:02}. 
     
     Finally, we update $s^{n+1}$ by 
\begin{equation}\label{update_s}
    \begin{aligned}
        s^{n+1} - s^n = \int_{\Omega}g(\mathcal{C}[\varphi_{NM}^{n+\frac{1}{2}}],s^{n+\frac{1}{2}})f(\mathcal{C}[\varphi_{NM}^{n+\frac{1}{2}}])\mathcal{C}[\varphi_{NM}^{n+1}-\varphi_{NM}^{n}]d\mathbf{x}.
    \end{aligned}
\end{equation}

     We summarize the fast rank-one algorithm in the following Algorithm \ref{fast algorithm_rank1}:

     \begin{algorithm}[htbp] 
	\caption{Fast rank-one algorithm for the scheme \eqref{ESAV_second_order_fully}}    
	 \label{fast algorithm_rank1}  
      \renewcommand{\algorithmicrequire}{\textbf{Input:}}
    \renewcommand{\algorithmicensure}{\textbf{Output:}}
	\begin{algorithmic}[1]
    	\Require $\mathbf{A},\, \mathbf{B},\mathbf{C},\mathbf{D},\,\mathbf{U}^{n},\mathbf{F}$, $s^n$, $s^{n-1}$ 
     \Ensure  $\mathbf{U}^{N_T+1}$ 
     \State  Calculate $\mathbf{A}^{-1}$,$\rm{W}^T$, $\beta = \rm{W}^T(\mathbf{A}^{-1}\mathbf{1}_{M})$
    \For {$n=0,\cdots,N_T$ }
    \State Compute $ \mathbf{H}^{n}_{km} = \left(\mathbf{D}\mathbf{\widehat{U}}^{n} + \mathbf{F}\right)_{km}$
        \State Compute $\alpha_{km} = \rm{W}^T\left(\mathbf{A}^{-1}\mathbf{H}^{\mathit{n}}_{\mathit{km}}\right)$
        \State Update $ \gamma^{\mathit{n}+1}_{km} = \frac{ - \mathbf{c}_{km}\beta}{1 + \mathbf{c}_{km}\beta}\gamma_{km}^n + \frac{\alpha_{km}}{1 + \mathbf{c}_{km}\beta}$
        \State Compute $\mathbf{\widehat{U}}_{km}^{n+1} = \mathbf{A}^{-1}\mathbf{H}^{n}_{km}- \mathbf{c}_{km}\gamma^{n+1}_{km}\mathbf{A}^{-1}\mathbf{1}_{M}- \mathbf{c}_{km}\gamma_{km}^{n}\mathbf{A}^{-1}\mathbf{1}_{M}$
        \State Recover $\mathbf{U}^{n+1}$ by the inverse FFT
        \State Update $s^{n+1}$ by scalar auxiliary equation \eqref{update_s}
 \EndFor   
        \Return $\mathbf{U}^{N_T+1}$
   \end{algorithmic} 
\end{algorithm}

\begin{remark}
The efficiency of the proposed fast rank-one algorithm originates from the
rank-one structure of the coupling matrix $\mathbf{B}$, which can be
expressed as the product of a column vector and a row vector. The rank-one structure is a direct consequence of the independent-mode decomposition guaranteed by Bernstein's theorem (Remark \ref{rmk:beyond_fractional}): all modes couple only through the single scalar $C[\varphi]$. 
The decomposition \eqref{fast_rank1:01}--\eqref{fast_rank1:04} is precisely the Sherman--Morrison formula \cite{sherman1950adjustment} applied to
the rank-one perturbation $\mathbf{A} + \mathbf{c}_{km}\mathbf{1}_M \mathrm{W}^T$.
\end{remark}

\begin{remark}
    It is observed that the proposed fast rank-one algorithm requires solving $M+1$ independent discrete Poisson-type equations and only involves \emph{sparse} matrix multiplications per time step. Importantly, the proposed fast rank-one algorithm avoids both characteristic decomposition and dense matrix multiplications. Consequently, for a fixed value of $M$, the computational cost and storage requirement of Algorithm \ref{fast algorithm_rank1} for the present scheme \eqref{ESAV_second_order_fully} are $\mathcal{O}\left(N_{T}N^{2}\log(N^{2})\right)$ and $\mathcal{O}(N^{2})$, respectively. 
\end{remark}

\subsection{Discrete structure-preserving properties}\label{sec:4}
In this subsection, we prove the discrete structure-preserving properties, i.e., the energy dissipation law and the mass conservation law, for the second-order fully discrete scheme. 

\begin{theorem}
    For $S\geq 0$, $\Delta t>0$, the  scheme \eqref{ESAV_second_order_fully} is unconditionally energy stable in the sense that the following discrete energy dissipation law holds:
\begin{equation}\label{Energy_dissipation_law_second_order}
\begin{aligned}
    \widehat{\mathcal{E}}(\varphi^{n+1}_{NM}) - \widehat{\mathcal{E}}(\varphi^{n}_{NM})\leq -\frac{1}{\Delta t}\left\Vert\mathcal{G}^{-1/2}(\varphi_{NM}^{n+1} - \varphi_{NM}^{n})\right\Vert^2_{L^{2}(\Omega;L^{2}_{\omega_{\alpha,0}}(I))},\quad n\geq 1,
\end{aligned}
\end{equation}
where $\widehat{\mathcal{E}}(\varphi^{n+1}_{NM}) =  \mathcal{E}(\varphi_{NM}^{n+1}) + \frac{S}{2}\left\Vert\mathcal{G}^{1/2}(\varphi_{NM}^{n+1}-\varphi_{NM}^{n})\right\Vert^2_{L^2(\Omega;L^2_{\omega_\alpha}(I))}$, and $\mathcal{E}(\varphi)$ is defined in \eqref{energy of AC}.
\end{theorem}
\begin{proof}
    Taking $q_{NM} = \varphi_{NM}^{n+1} - \varphi_{NM}^{n}$ in \eqref{ESAV_second_order_fully}, we obtain
    \begin{equation*}
        \begin{aligned}
            &\frac{1}{\Delta t}\left(\mathcal{G}^{-1}(\varphi^{n+1}_{NM} - \varphi^{n}_{NM}),\varphi_{NM}^{n+1} - \varphi_{NM}^{n}\right)_{L^{2}(\Omega;L^{2}_{\omega_{\alpha,0}}(I))} \\
            &+ \frac{1}{2}\left(\mathcal{G}^{-1}(\varphi^{n}_{NM} + \varphi^{n+1}_{NM}-2\phi_{0}),\varphi_{NM}^{n+1} - \varphi_{NM}^{n}\right)_{L^{2}(\Omega;L^{2}_{\omega_{\alpha,1}}(I))} \\
        & = - \frac{\varepsilon^{2}}{2}\left(\nabla(\mathcal{C}[\varphi_{NM}^{n} + \varphi_{NM}^{n+1}]),\nabla(\mathcal{C}[\varphi_{NM}^{n+1}-\varphi_{NM}^{n}])\right)_{L^{2}(\Omega)} \\
        &- g\left(\mathcal{C}[\varphi_{NM}^{n+\frac{1}{2}}],s^{n+\frac{1}{2}}\right)\left(f(\mathcal{C}[\varphi_{NM}^{n+\frac{1}{2}}]),\mathcal{C}[\varphi_{NM}^{n+1}-\varphi_{NM}^{n}]\right)_{L^{2}(\Omega)}\\
        &- S\left(\mathcal{G}(\varphi_{NM}^{n+1}-2\varphi_{NM}^{n}+\varphi_{NM}^{n-1}),\varphi_{NM}^{n+1}-\varphi_{NM}^{n}\right)_{L^2(\Omega;L^2_{\omega_{\alpha}}(I))},\\
            &s^{n+1} -s^{n} = g\left(\mathcal{C}[\varphi_{NM}^{n+\frac{1}{2}}],s^{n+\frac{1}{2}}\right)\left(f(\mathcal{C}[\varphi^{n+\frac{1}{2}}_{NM}]), \mathcal{C}[\varphi^{n+1}_{NM} - \varphi^{n}_{NM}]\right)_{L^{2}(\Omega)}.
        \end{aligned}
    \end{equation*}
    Noting  the identity $2(a-2b+c)(a-b) = (a-b)^{2} - (b-c)^{2} + (a-2b+c)^{2}$, then we obtain the following equation by combining the above equation:
    \begin{equation*}
        \begin{aligned}
             &\left( \mathcal{E}(\varphi_{NM}^{n+1}) + \frac{S}{2}\left\Vert\mathcal{G}^{1/2}(\varphi_{NM}^{n+1}-\varphi_{NM}^{n})\right\Vert^2_{L^2(\Omega;L^2_{\omega_\alpha}(I))} \right)\\
             &- \left(\mathcal{E}(\varphi_{NM}^{n}) + \frac{S}{2}\left\Vert\mathcal{G}^{1/2}(\varphi_{NM}^{n}-\varphi_{NM}^{n-1})\right\Vert^2_{L^2(\Omega;L^2_{\omega_\alpha}(I))}\right)\\   
                &= -\frac{1}{\Delta t}\left\Vert\mathcal{G}^{-1/2}(\varphi_{NM}^{n+1} - \varphi_{NM}^{n})\right\Vert^2_{L^{2}(\Omega;L^{2}_{\omega_{\alpha,0}}(I))}\\
                &- \frac{S}{2}\left\Vert\mathcal{G}^{1/2}(\varphi_{NM}^{n+1}-2\varphi_{NM}^{n}+\varphi_{NM}^{n-1})\right\Vert^2_{L^{2}(\Omega;L^{2}_{\omega_{\alpha}}(I))},
        \end{aligned}
    \end{equation*}
    which implies the desired result \eqref{Energy_dissipation_law_second_order} after neglecting the non-positive term.
\end{proof}

\begin{theorem}
    For $S\geq0$, $\Delta t>0$, assume that $\varphi_{NM}^{n+1}$ is the solution of the second-order scheme \eqref{ESAV_second_order_fully}, then we have the following discrete global mass-conservation and pointwise (in $\theta$) mass-conservation: 
    \begin{equation}\label{Mass_result:second_order}
    \begin{aligned}
       &\int_\Omega\mathcal{C}[\varphi_{NM}^{n+1}](\mathbf{x})d\mathbf{x} = \int_\Omega\mathcal{C}[\varphi_{NM}^{n}](\mathbf{x})d\mathbf{x},\; n = 0,\cdots,N_T-1,\\
       &\int_{\Omega}\varphi_{NM}^{n+1}(\mathbf{x},\theta)d\mathbf{x} = \int_{\Omega}\varphi_{NM}^{n}(\mathbf{x},\theta)d\mathbf{x},\; \theta\in(0,1),\; n = 0,\cdots,N_T-1.
    \end{aligned}
    \end{equation} 
\end{theorem}
\begin{proof}

It follows from \eqref{ESAV_second_order_fully} that:
\begin{equation*}
    \begin{aligned}
         &\frac{1}{\Delta t}\left(\varphi_{NM}^{n+1} - \varphi_{NM}^{n},q_{NM}\right)_{L^2(\Omega)} + \frac{c_1(\theta)}{2}\left(\varphi_{NM}^{n+1}+\varphi_{NM}^{n} - 2\phi_{0},q_{NM}\right)_{L^2(\Omega)}\\
    &=c_0(\theta) \left(-\frac{\varepsilon^{2}}{2}(-\Delta) \left(\mathcal{C}[\varphi_{NM}^{n+1}+\varphi_{NM}^{n}]\right), (-\Delta) q_{NM}\right)_{L^2(\Omega)}\\
    & +c_0(\theta)\left( g(\mathcal{C}[\varphi_{NM}^{n+\frac{1}{2}}],s^{n+\frac{1}{2}})f(\mathcal{C}[\varphi^{n+\frac{1}{2}}_{NM}]),(-\Delta) q_{NM}\right)_{L^2(\Omega)}\\
&- S c_0(\theta)\left((-\Delta)(\varphi_{NM}^{n+1}-2\varphi_{NM}^{n}+\varphi_{NM}^{n-1}),(-\Delta)q_{NM}\right)_{L^2(\Omega)}.
    \end{aligned}
\end{equation*}
    Taking $q_{NM} = 1$ in the above equation, since $(-\Delta)q_{NM} = 0$ for $q_{NM} = 1$, then we obtain
\begin{equation}\label{mass:05}
         \frac{1}{\Delta t}\left(\varphi^{n+1}_{NM} - \varphi^n_{NM},1\right)_{L^2(\Omega)} + \frac{ c_1(\theta)}{2}\left(\varphi^{n+1}_{NM} + \varphi^{n}_{NM} - 2\phi_0,1\right)_{L^2(\Omega)} = 0.
    \end{equation}
    We prove $\left(\varphi_{NM}^{n+1},1\right)_{L^2(\Omega)}= \left(\phi_0,1\right)_{L^2(\Omega)}$ by using mathematical induction method. In fact, for $n = 0$, we obtain
    \begin{equation*}
        \left(\frac{1}{\Delta t} + \frac{c_1(\theta)}{2}\right)\left(\varphi_{NM}^{1} - \phi_0,1\right)_{L^2(\Omega)}= 0.
    \end{equation*}
    We assume $\left(\varphi_{NM}^{s},1\right) _{L^2(\Omega)} = \left( \phi_0,1\right)_{L^2(\Omega)}$, $\forall~ 1\leq s\leq n$, then we obtain for $s = n+1$ 
    \begin{equation*}
    \begin{aligned}
         \left(\frac{1}{\Delta t} + \frac{c_1(\theta)}{2}\right)\left(\varphi_{NM}^{n+1},1\right) _{L^2(\Omega)}&= \left(\frac{1}{\Delta t} - \frac{c_1(\theta)}{2}\right)\left(\varphi_{NM}^{n},1\right) _{L^2(\Omega)}+ c_1(\theta)\left(\phi_0,1\right) _{L^2(\Omega)}\\
         &= \left(\frac{1}{\Delta t} + \frac{c_1(\theta)}{2}\right)\left(\phi_0,1\right) _{L^2(\Omega)},
    \end{aligned}
    \end{equation*}
    that is 
    \begin{equation}\label{mass:06}
        \left(\frac{1}{\Delta t} + \frac{c_1(\theta)}{2}\right)\left(\varphi_{NM}^{n+1}-\phi_0,1\right) _{L^2(\Omega)}=0.
    \end{equation}
    Taking the $L^2_{\omega_{\alpha}}(I)$ inner product of \eqref{mass:05} yields
        \begin{equation}\label{mass:07}
        \frac{1}{\Delta t}\left(\varphi^{n+1}_{NM}-\varphi^{n}_{NM},1\right)_{L^2(\Omega;L^2_{\omega_{\alpha}}(I))}+ \left(\frac{c_1(\theta)}{2}(\varphi^{n+1}_{NM} + \varphi_{NM}^{n} - 2\phi_0),1\right)_{L^2(\Omega ;L^2_{\omega_{\alpha}}(I))} = 0.
    \end{equation}   
    Thus, the mass-conservation \eqref{Mass_result:second_order} follows from \eqref{mass:05}-\eqref{mass:07}.
\end{proof}

\section{Numerical results}\label{sec:6}
In this section, we present numerical experiments for both the Allen--Cahn case ($\mathcal{G} = I$) and the Cahn--Hilliard case ($\mathcal{G} = -\Delta$). Examples 1--3 verify the structure-preserving
properties of the numerical schemes for both cases. Example~4 demonstrates
the coalescence dynamics of the Cahn--Hilliard case, while Example~5 illustrates the physical
effect of memory on coarsening dynamics for the Allen--Cahn case, directly connecting to the
friction mechanism derived in \S\ref{sec:derivation}.


%
\begin{example}[\textbf{Accuracy test}]\label{example1} 
In this example, we consider the following equation:
\begin{equation*}
    \frac{\partial ^{\alpha}\phi}{\partial t^{\alpha}}(x_1,x_2,t) = -\mathcal{G}\left[-
     \varepsilon^{2} \Delta \phi(x_1,x_2,t) + f(\phi(x_1,x_2,t))\right] + g(x_1,x_2,t).
\end{equation*}
In this example, we set   $\varepsilon = 0.1$, $T = 1$, $\Omega = [0,2\pi)^2$. In order to test the accuracy of temporal and the internal-variable $\theta$, the following two cases are considered:
\begin{itemize}
    \item Case I: The exact solution $\phi(x_1,x_2,t) = 0.02t^2 \cos(x_1)\cos(x_2)$;
    \item  Case II: The source term $g(x_1,x_2,t) = 0$, the initial value $\phi_{0}(x_1,x_2) = 0.02\cos(x_1)\cos(x_2)$.
\end{itemize}
\end{example}
\begin{figure}[htbp]
	\centering
    \subfigure[Temporal convergence results\label{1a}] 
	{
		\begin{minipage}{0.35\linewidth}
			\centering
			\includegraphics[width=\linewidth]{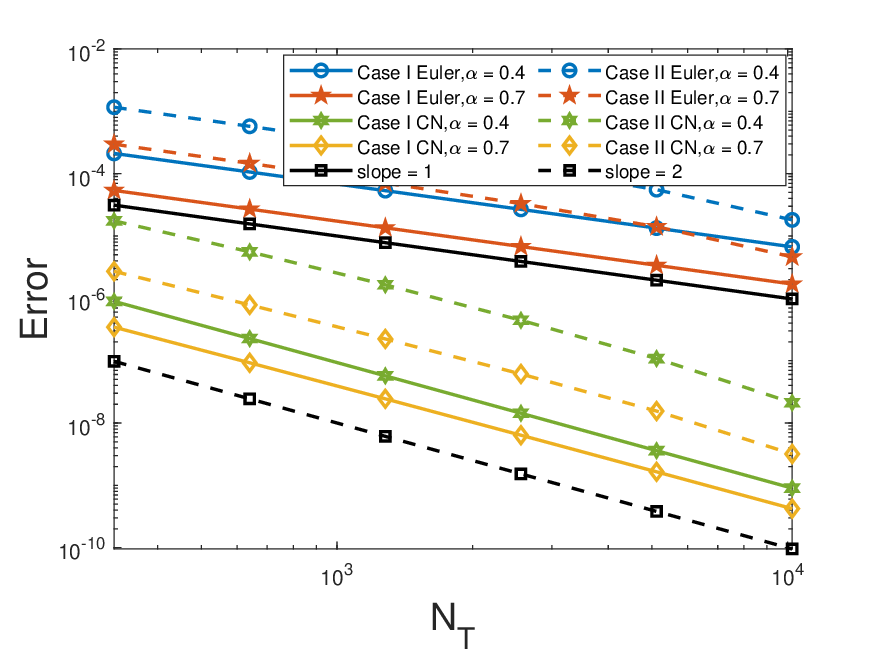}
		\end{minipage}
	}
    \subfigure[$\theta$ convergence results\label{1b}] 
	{
		\begin{minipage}{0.35\linewidth}
			\centering
			\includegraphics[width=\linewidth]{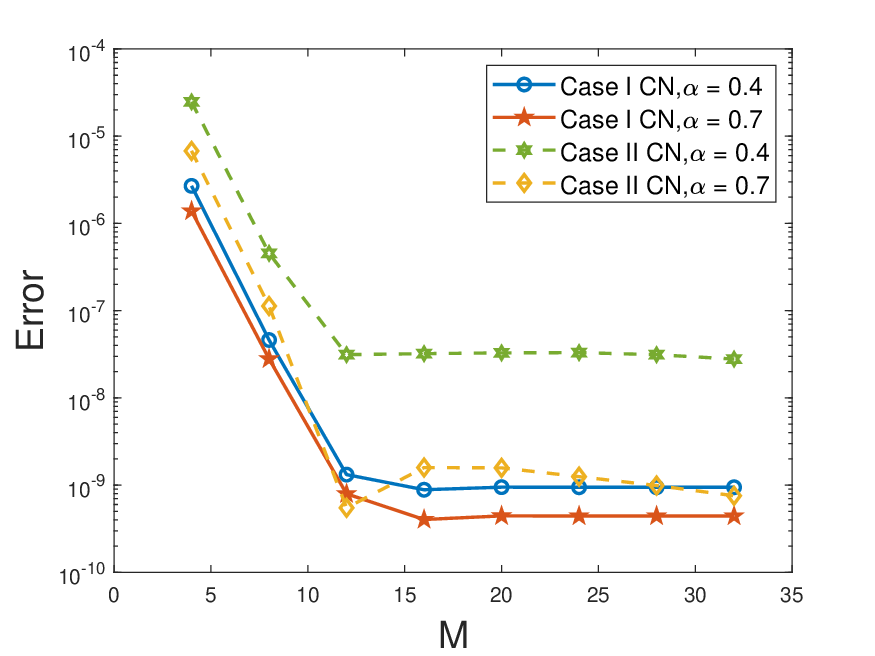}
		\end{minipage}
	}
	
 \caption{Example \ref{example1}. \ref{1a}: The maximum error vs. the time-nodes $N_{T}$ produced by the proposed first-order scheme \eqref{ESAV_first_order_fully} and second-order scheme \eqref{ESAV_second_order_fully} for Case I and Case II with the different values of $\alpha$. \ref{1b}:The maximum error vs. the Gauss nodes $M$ produced by the proposed second-order scheme \eqref{ESAV_second_order_fully} for Case I and Case II with  $\alpha = 0.3,\, 0.7$ and $\Delta t = 10^{-4}$.} 
\end{figure}

In this example, we present only the convergence results for $\mathcal{G} = I$ for simplicity, as similar results for $\mathcal{G} = -\Delta$ can be obtained and are omitted here. In Case II, we use the numerical solution computed with a sufficiently small time-step as the reference solution. We first investigate the temporal accuracy for the different values of fractional order $\alpha$. To this end, we set $M = 30$ for $\theta $ direction and $N = 128$ for the spatial discretization. The convergence results of the time-discretization for both Case I and Case II with the different values of fractional order $\alpha$ at $T = 1$ are shown in Figure \ref{1a}, which confirms that the expected convergence orders are obtained.

Then, we demonstrate the spectral accuracy of the approximation for the internal variable $\theta$. Here, $\Delta t$ is set to a sufficiently small value. Figure \ref{1b} shows  the accuracy with respect to $M$ in a semi-log scale for the different values of fractional order $\alpha = 0.3,\, 0.7$ at time $T = 1$.

It can be observed that \emph{exponential convergence} rates are achieved in all tests for the internal-variable $\theta$. This finding verifies the regularity result established in \cite[Theorem 5]{DingMao2024}, which states that the solution possesses sufficiently high regularity with respect to $\theta$ to guarantee the spectral accuracy. This indicates that we can fix \emph{a small number} of nodes (i.e., a small value of $M$) for the spectral discretization. Consequently, by employing the fast rank-one algorithm (see Algorithm \ref{fast algorithm_rank1}), we only need to solve \emph{a few independent discrete Poisson equations, which only involve sparse matrix multiplications}.
\begin{figure}[htbp]
	\centering 
	{
		\begin{minipage}{0.35\linewidth}
			\centering
			\includegraphics[width=\linewidth]{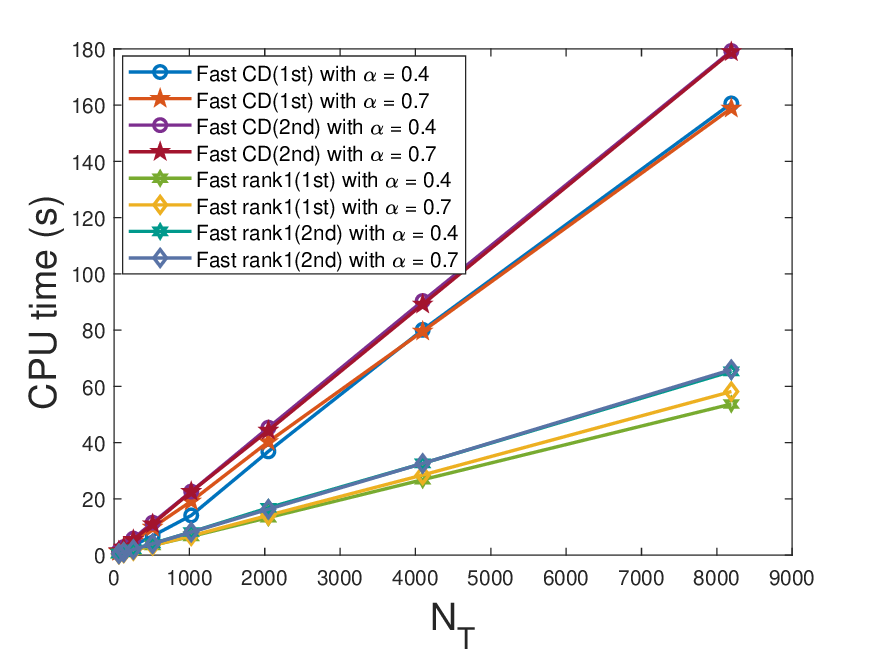}
		\end{minipage}
	}
	
 \caption{ Example \ref{example1}.  The CPU times of our fast rank-one algorithm ( Algorithm \ref{fast algorithm_rank1}) and the proposed fast algorithm in \cite{DingMaoLiu} with the different values of the fractional order $\alpha$.}
 \label{Efficiency_Fast_rank1}
\end{figure}

To demonstrate the computational efficiency of the proposed algorithm (Algorithm \ref{fast algorithm_rank1}), we compare in Figure \ref{Efficiency_Fast_rank1} the computation time  of the present fast rank-one algorithm with those of a fast algorithm based on generalized characteristic decomposition techniques from our previous work \cite{DingMaoLiu}. For Case I, we set $N = 64$, $M = 30, \,\varepsilon = 0.5$. As observed in Figure \ref{Efficiency_Fast_rank1},  the CPU time of the proposed fast rank-one algorithm is significantly lower than that of the fast characteristic decomposition  algorithm, demonstrating its superior computational efficiency. This is because our algorithm exploits the rank-one structure of matrix $\mathbf{B}$ instead of using characteristic decomposition. Consequently, it only requires solving $M+1$ Poisson-type equations and involves only \emph{sparse} matrix multiplications. Therefore, the computational cost and storage requirements of our fast rank-one algorithm ( Algorithm \ref{fast algorithm_rank1}) are $\mathcal{O}(N_TN^2\log(N^2))$ and $\mathcal{O}(N^2)$ for fixed $M$, respectively.

\begin{figure}[htbp]
	\centering
 \subfigure[$\mathcal{G}=I$\label{5b}]
	{
		\begin{minipage}{0.35\linewidth}
			\centering    
			\includegraphics[width=\linewidth]{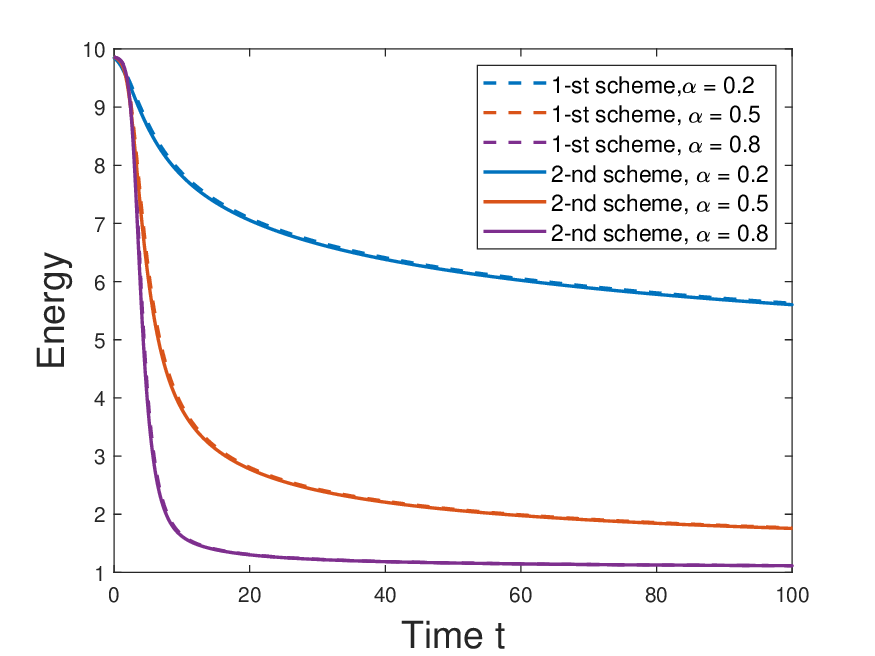} 
		\end{minipage}
	}
    \subfigure[$\mathcal{G} = -\Delta$\label{5c}] 
	{
		\begin{minipage}{0.35\linewidth}
			\centering
			\includegraphics[width=\linewidth]{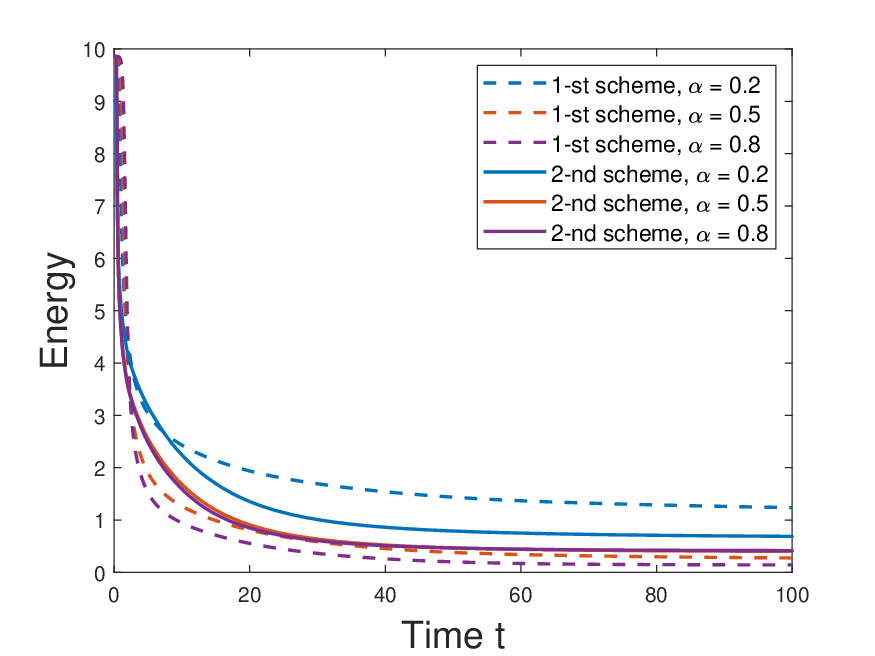}
		\end{minipage}
	}

 \caption{Example \ref{example2}. Energy evolution of the simulated solution for $\mathcal{G}=I$ (left) and $\mathcal{G}=-\Delta$ (right)  with the different values of fractional order $\alpha$ and $\Delta t = 0.01$. } 
	 \label{Energy_dissipation_CN}
\end{figure}
\begin{example}[\textbf{Energy dissipation law}]\label{example2} In this example, we consider \eqref{EPDEs} with $\varepsilon = 0.01$. We take $N = 128$ and $M = 30$. We set the initial value by random numbers in $[-0.04,0.04]$ and time step size $\Delta t = 0.01$, stabilization parameter $S = 2$.  
\end{example}

The time evolutions of the energy for the first- and second-order schemes for both $\mathcal{G} = I$ and $\mathcal{G} = -\Delta$ with the different values of $\alpha$ are presented in  Figure \ref{Energy_dissipation_CN}. We observe that the energy  decreases at all times for the different values of $\alpha$, which is consistent with our theoretical result (i.e., Theorem \ref{energy dissipation for AC}).
\begin{example}[\textbf{Mass-conservation}]\label{example3} In this example, we investigate the mass-conservation for the Cahn--Hilliard case with the following initial value
 \begin{equation*}
     \phi_0(x_1,x_2) = \cos(x_1)\cos(x_2).
 \end{equation*}
\end{example}
We take $T = 100$, $N = 128$, $S = 2$, $M = 30$ and $\varepsilon = 0.1$, 
and define the following global mass error and pointwise mass error between the numerical solution $\varphi^{n+1}_{NM}$ and the initial value $\phi_0$:
\begin{equation*}
\begin{aligned}
     &GME(t_n) = \int_{\Omega}\mathcal{C}[\varphi^{n}_{NM} - \phi_0](\mathbf{x}) d\mathbf{x},\quad 0\leq n\leq N_{T},\\
     & PME(t_n,\theta_s) = \int_{\Omega}\left(\varphi^{n}_{NM}(\mathbf{x},\theta_s) - \phi_0(\mathbf{x})\right) d\mathbf{x},\quad 0\leq n\leq N_{T}, \theta_s\in (0,1).
     \end{aligned}
\end{equation*}

We present the plots of the global mass error and the pointwise mass error for the first- and the second-order schemes with $\alpha = 0.2,\,0.8$ in Figure \ref{2D}. We observe in Figure \ref{2D} that the global mass error (GME) and the pointwise mass error (PME) are reduced to the machine precision ($\mathcal{O}(10^{-16})$) for both the first- and the second-order schemes, which is consistent with Theorem \ref{mass_preserving}. 

To demonstrate the pointwise mass conservation with respect to $\theta$, we present a 3D pointwise mass error surface plot of the second-order scheme with $\alpha = 0.5$ in Figure \ref{3D}.  We observe in Figure \ref{3D} that the  pointwise mass error (PME) can reach $\mathcal{O}(10^{-16})$ for each $\theta\in (0,1)$. This demonstrates that our numerical method maintains not only global mass conservation but also pointwise mass conservation. The latter provides a stronger and more intrinsic constraint than global mass conservation for developing numerical schemes with higher fidelity.

\begin{figure}[htbp]
	\centering
 \subfigure[2D\label{2D}]	{ {
		\begin{minipage}{0.45\linewidth}
			\centering
			\includegraphics[width=\linewidth]{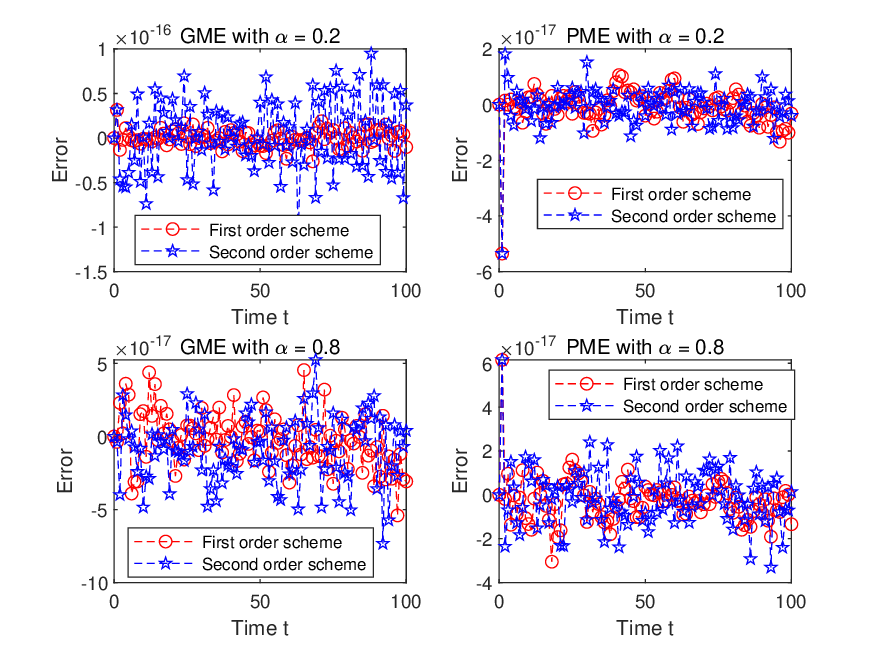}
		\end{minipage}
	}
    }
    \subfigure[3D\label{3D}]{   
    
    {
		\begin{minipage}{0.45\linewidth}
			\centering
			\includegraphics[width=\linewidth]{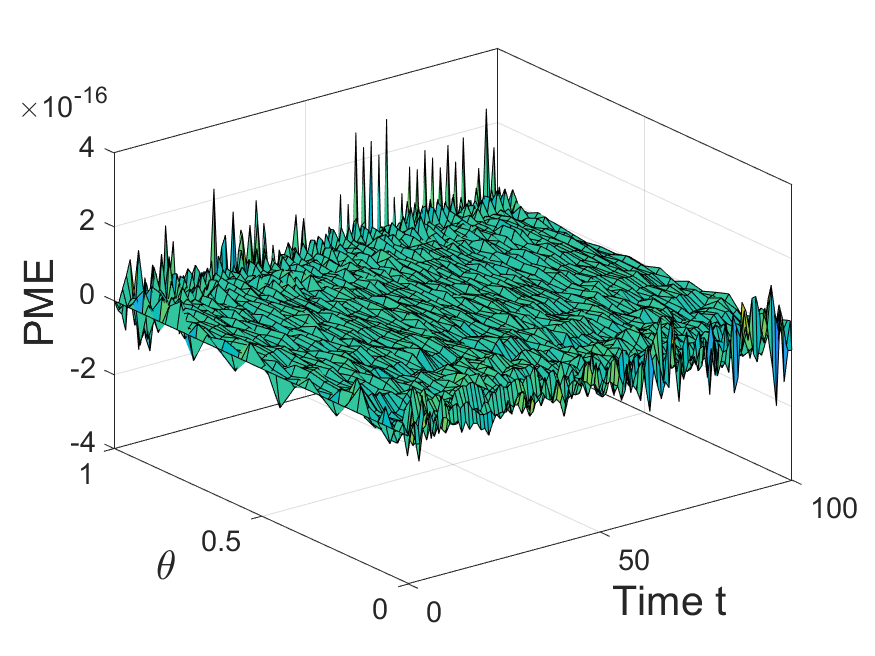}
		\end{minipage}
	}    
    }
    \caption{Example \ref{example3}. \ref{2D}:The global mass error (left) and pointwise mass error (right) for both the first- and second-order schemes with $\alpha = 0.2,\,0.8$. Here, we take the 20th layer along the $\theta$ direction of the numerical solution $\varphi_{NM}^n(\mathbf{x},\theta)$ to compute PMC. \ref{3D}:The 3D pointwise mass error surface plot of second-order scheme with $\alpha = 0.5$.} 
\end{figure}

\begin{example}[\textbf{Two kissing bubbles}]\label{example4}
    We simulate the coalescence of two bubbles, where the initial value is given by
    \begin{equation*}
    \begin{aligned}
        \phi_{0}(x_1,x_2) &= \tanh\left(\frac{R_{1}^2}{\varepsilon^2} - \frac{1}{\varepsilon^2}\left((x_1 - a_{1})^2 + (x_2 - b_{1})^2\right)\right)\\
        &+ \tanh\left(\frac{R_{2}^2}{\varepsilon^2} - \frac{1}{\varepsilon^2}\left((x_1 - a_{2})^2 + (x_2 - b_{2})^2\right)\right) + 1.
    \end{aligned}   
    \end{equation*}
    We choose $a_{1}= b_1 = \pi -1$, $a_2 = b_2 = \pi+1$, $R_1 = R_2 = 1$, $\varepsilon = 0.05$.                         
\end{example}
\begin{figure}[htbp]
	\centering
	{
		\begin{minipage}{0.2\linewidth}
			\centering
			\includegraphics[width=\linewidth]{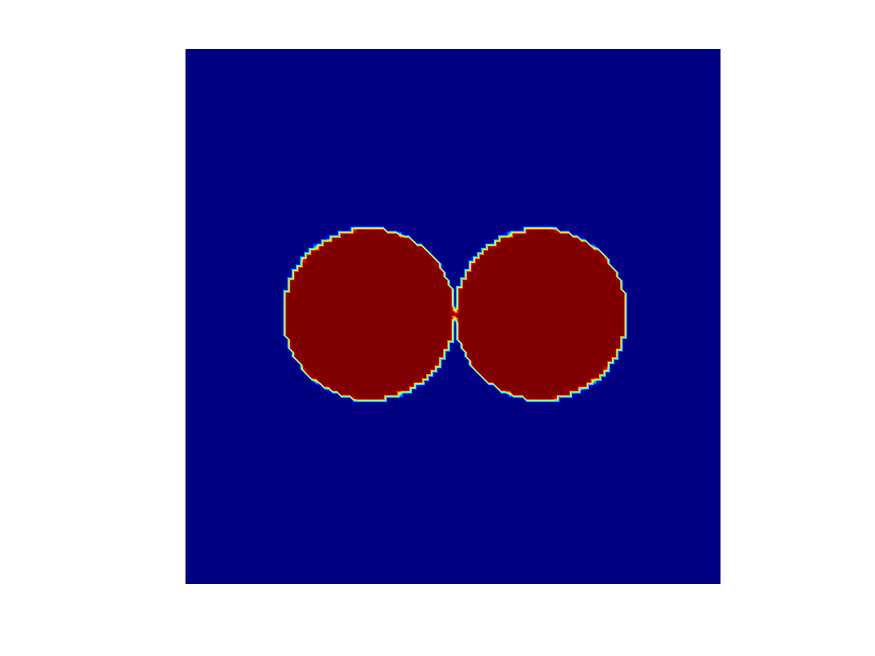}
		\end{minipage}
	}
	{
		\begin{minipage}{0.2\linewidth}
			\centering
			\includegraphics[width=\linewidth]{pic/Two/alpha9/0.eps}
		\end{minipage}
  }
	{
		\begin{minipage}{0.2\linewidth}
			\centering
			\includegraphics[width=\linewidth]{pic/Two/alpha9/0.eps}
		\end{minipage}
  }
  {
		\begin{minipage}{0.2\linewidth}
			\centering
			\includegraphics[width=\linewidth]{pic/Two/alpha9/0.eps}
		\end{minipage}
  }
  
  {
		\begin{minipage}{0.2\linewidth}
			\centering
			\includegraphics[width=\linewidth]{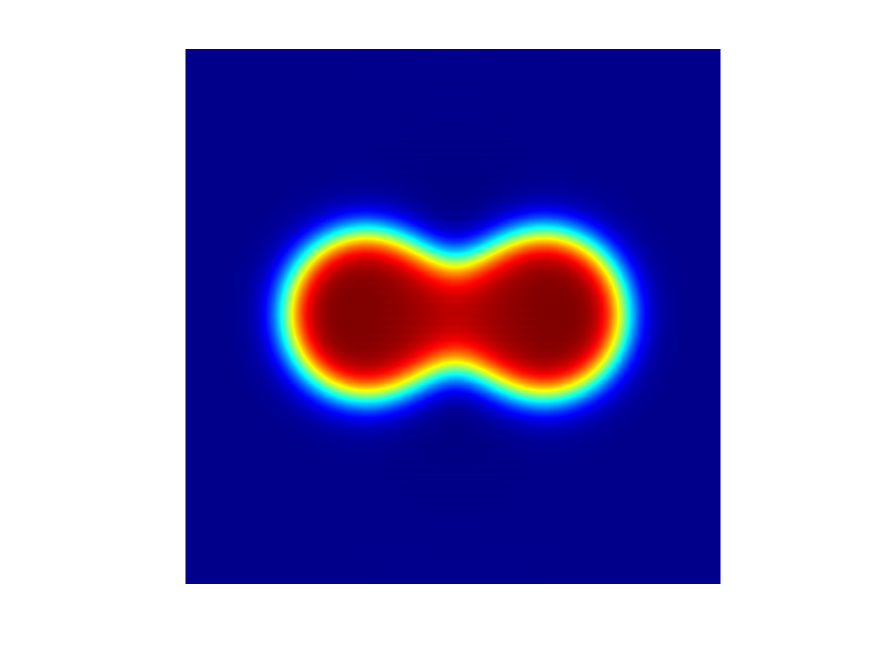}
		\end{minipage}
	}
	{
		\begin{minipage}{0.2\linewidth}
			\centering
			\includegraphics[width=\linewidth]{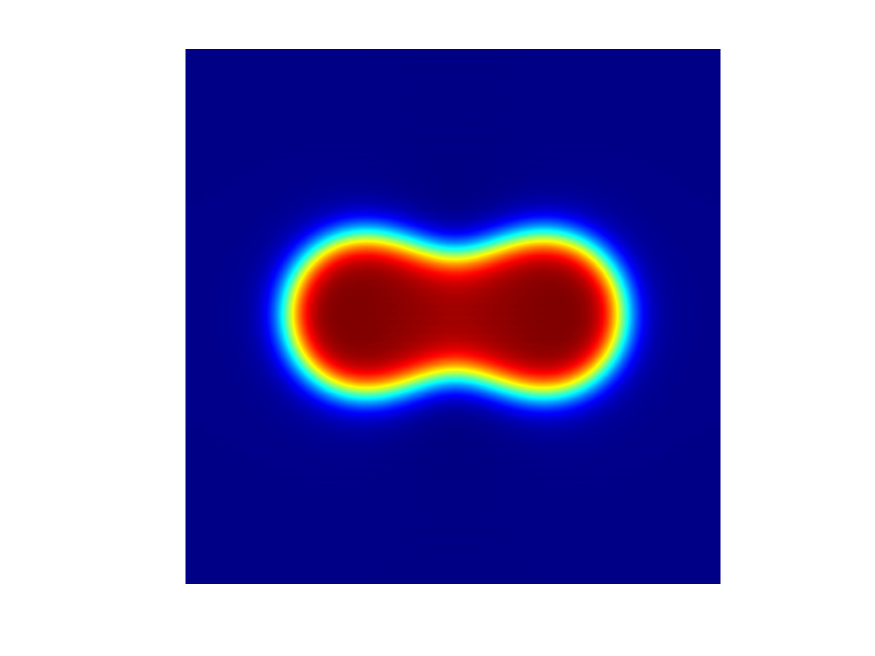}
		\end{minipage}
  }
	{
		\begin{minipage}{0.2\linewidth}
			\centering
			\includegraphics[width=\linewidth]{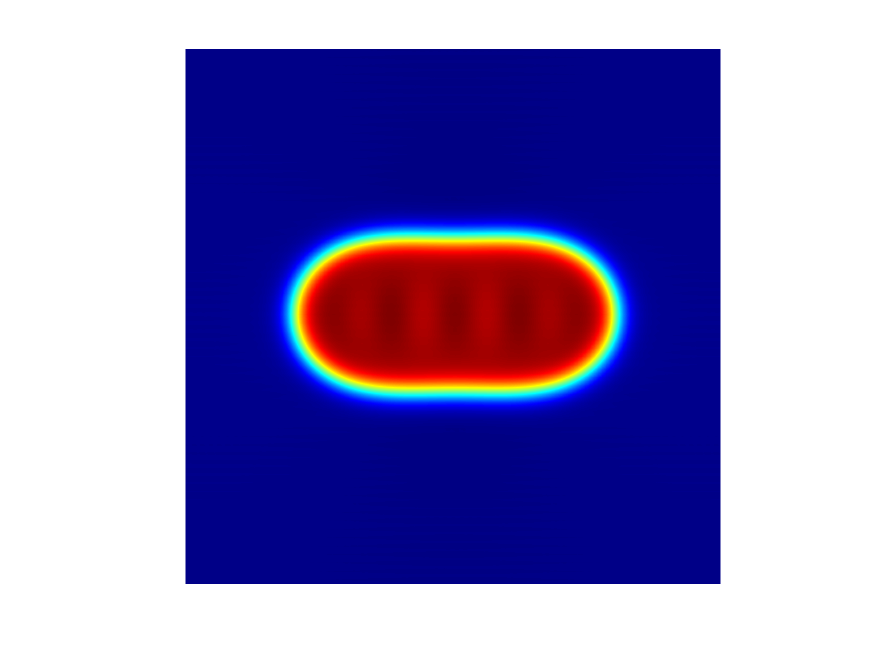}
		\end{minipage}
  }
  {
		\begin{minipage}{0.2\linewidth}
			\centering
			\includegraphics[width=\linewidth]{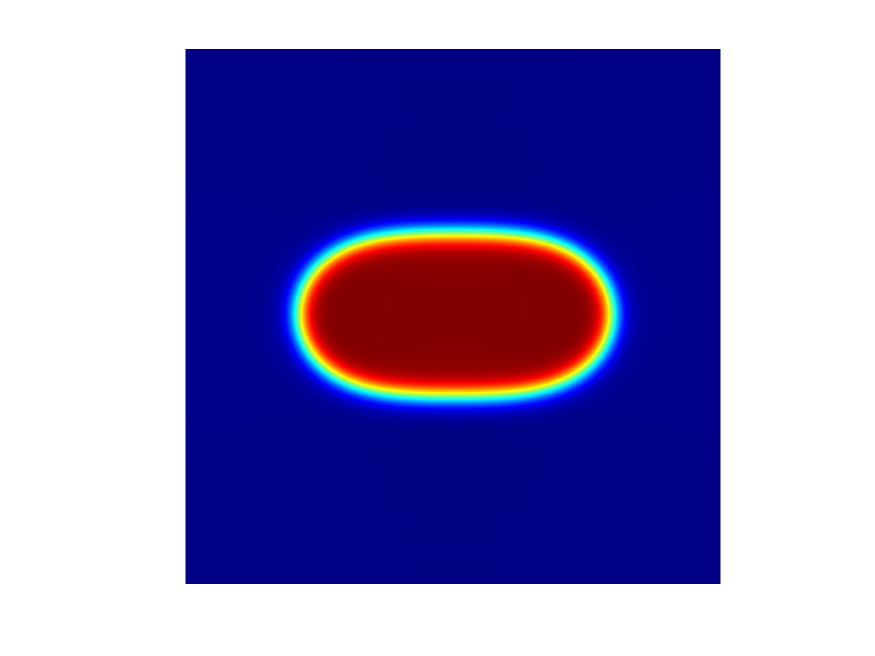}
		\end{minipage}
  }

   \subfigure[ $\alpha = 0.3$]
{
		\begin{minipage}{0.2\linewidth}
			\centering
			\includegraphics[width=\linewidth]{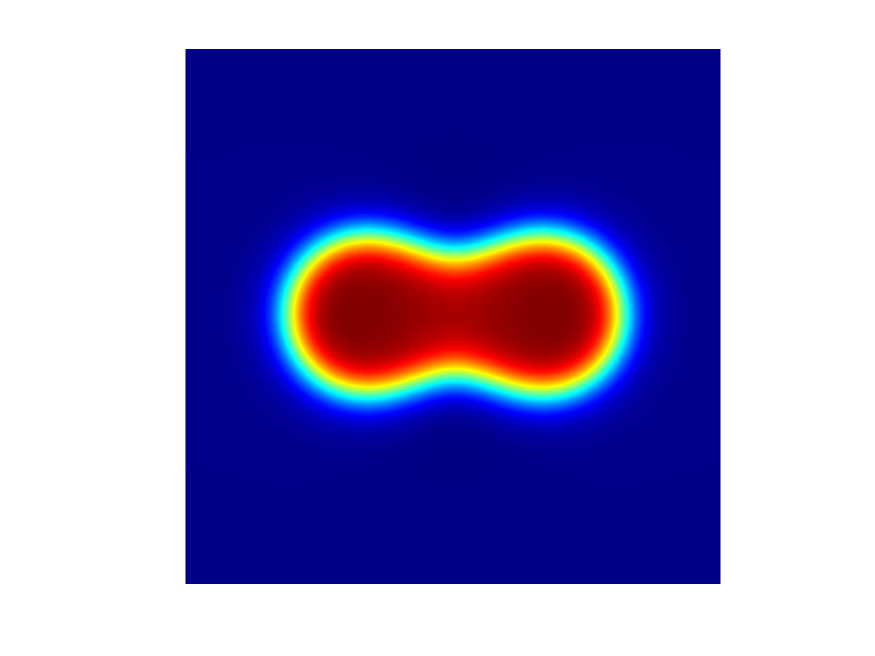}
		\end{minipage}
	}
    \subfigure[ $\alpha = 0.6$]
	{
		\begin{minipage}{0.2\linewidth}
			\centering
			\includegraphics[width=\linewidth]{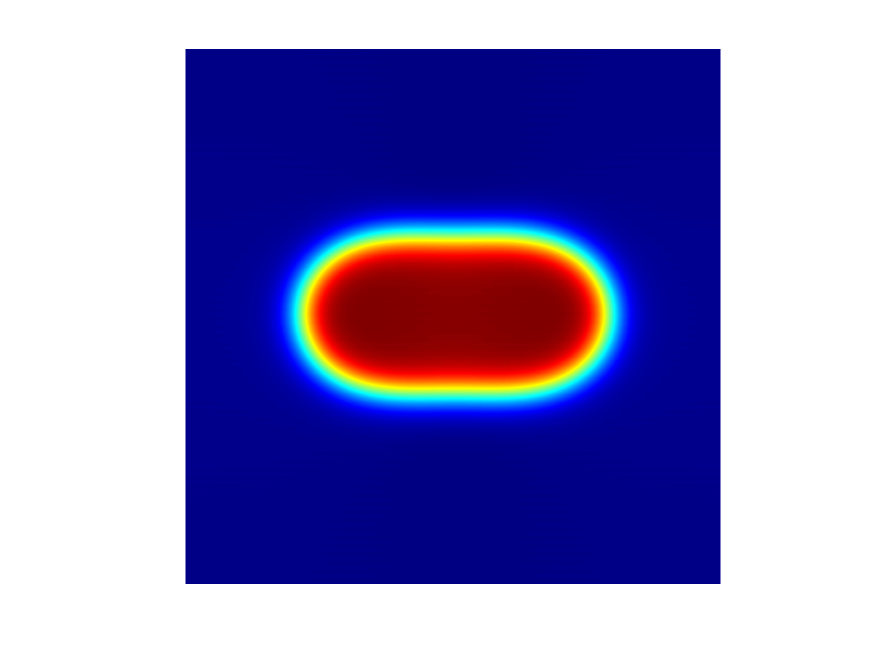}
		\end{minipage}
  }
  \subfigure[ $\alpha = 0.9$]
	{
		\begin{minipage}{0.2\linewidth}
			\centering
			\includegraphics[width=\linewidth]{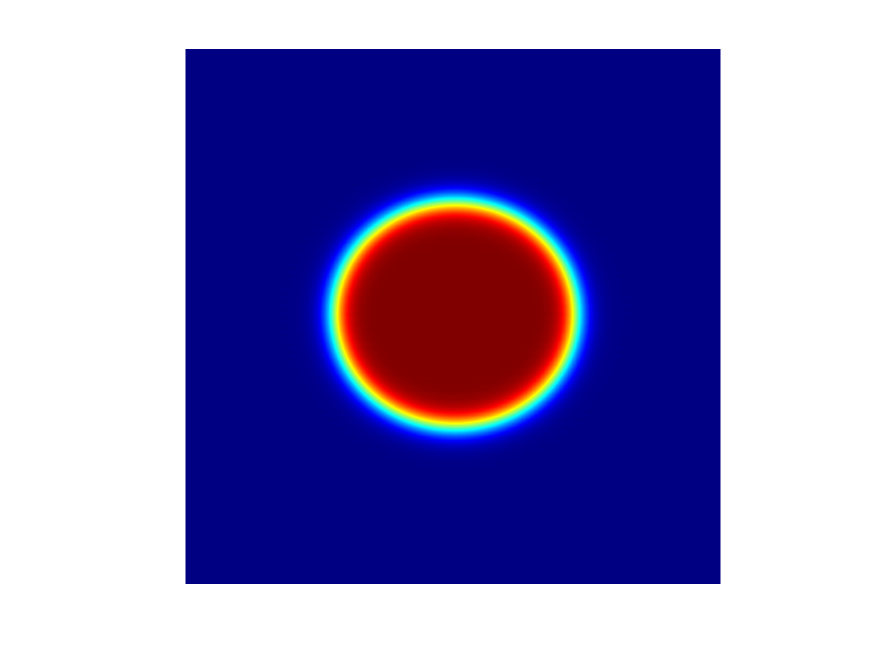}
		\end{minipage}
  }
  \subfigure[ $\alpha = 0.9999$\label{single_bubble2}]
	{
		\begin{minipage}{0.2\linewidth}
			\centering
			\includegraphics[width=\linewidth]{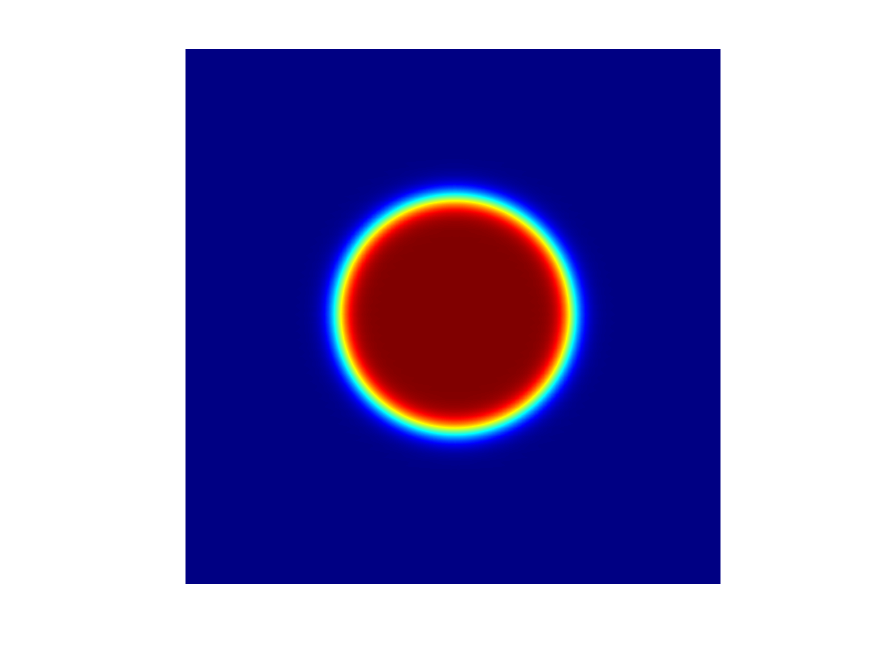}
		\end{minipage}
  }

	\caption{Example \ref{example4}. Evolution of the solution with $\mathcal{G} = -\Delta$ for the different values of fractional order $\alpha$. From top to bottom: $t = 0,10,200$.}
 \label{twobubbles_CH}
\end{figure}

 In order to illustrate the effect of changing the value of $\alpha$ on the character of the solution of \eqref{EPDEs}, we present the snapshots of the solution for different values of $\alpha$ at different time instances in Figure \ref{twobubbles_CH}. The numerical solution for the scheme \eqref{ESAV_first_order_fully} is obtained by using the Fourier-Galerkin spectral method  with $N = 128$, the Jacobi-collocation spectral method with $M = 30$ and the time-step size $\Delta t = 0.1$.

 We observe in Figure \ref{twobubbles_CH} that as time evolves, the two bubbles coalesce into a single bubble. Additionally, as $\alpha \rightarrow 1$, the classical Cahn--Hilliard equation preserves mass, and we see that the bubble reaches a steady circular shape (see Figure \ref{single_bubble2}).   More interestingly, the rate at which the energy decreases varies with  $\alpha$, i.e., the energy decays faster with increasing values of  $\alpha$, which is consistent with the energy dissipation law (see Figure \ref{Energy_dissipation_CN}).

 \begin{example}[\textbf{Coarsening dynamics}]\label{example5}
    In this example, we simulate the coarsening dynamics for the Allen-Cahn case with $\mathcal{G} = I$ to illustrate the physical effect of the memory kernel on pattern evolution. We take $N = 256$ and $M = 30$. We set 
    time step size $\Delta t = 0.1$. The initial value is  drawn from a uniform distribution on $[-0.4, 0.4]$. 
\end{example}

To better visualize the coarsening dynamics, we simulate $\phi$ using a relatively large thickness parameter $\varepsilon=0.05$. Its snapshots with the different values of $\alpha$ at $t = 0,\,10,\,20,\,50,\,200$
are shown in Figure \ref{Coarsening_Dynamics_for_AC}. We can see that the coarsening is initially slower for small value of $\alpha$.

 \begin{figure}[htbp]
\centering
\setlength{\tabcolsep}{1pt}     

\begin{adjustbox}{width=\textwidth}
\begin{tabular}{ccccc}

\includegraphics[width=0.198\textwidth]{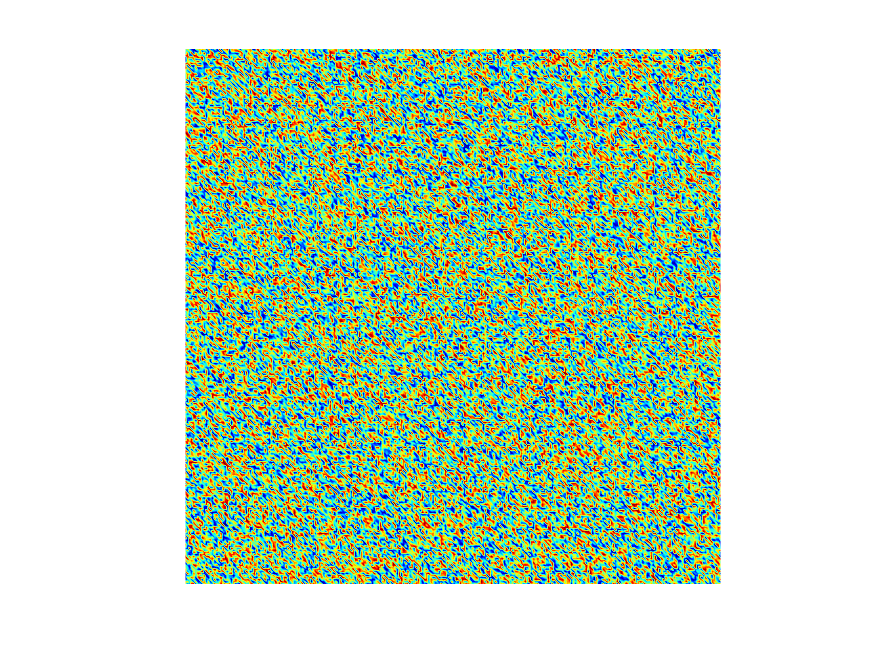} &
\includegraphics[width=0.198\textwidth]{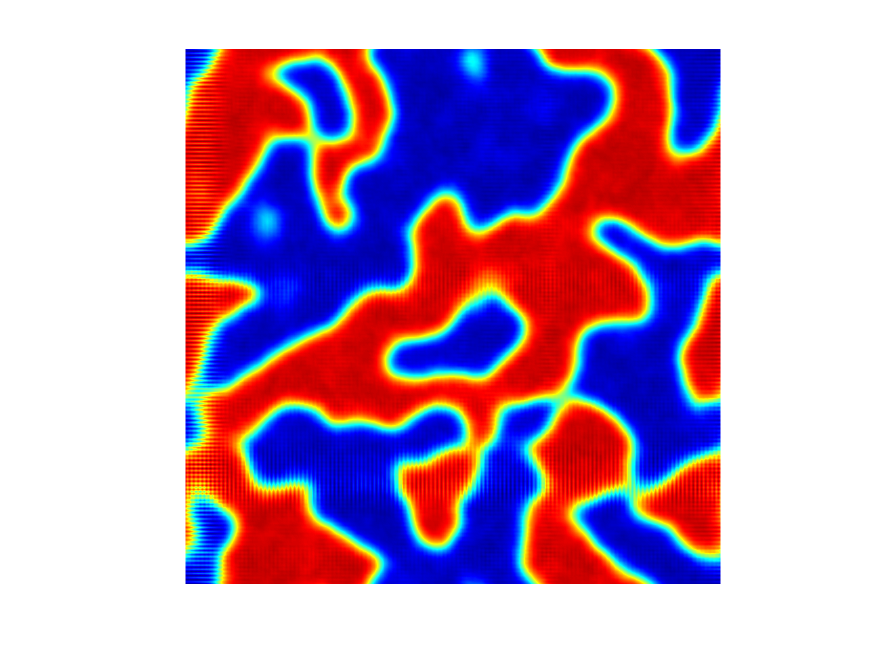} &
\includegraphics[width=0.198\textwidth]{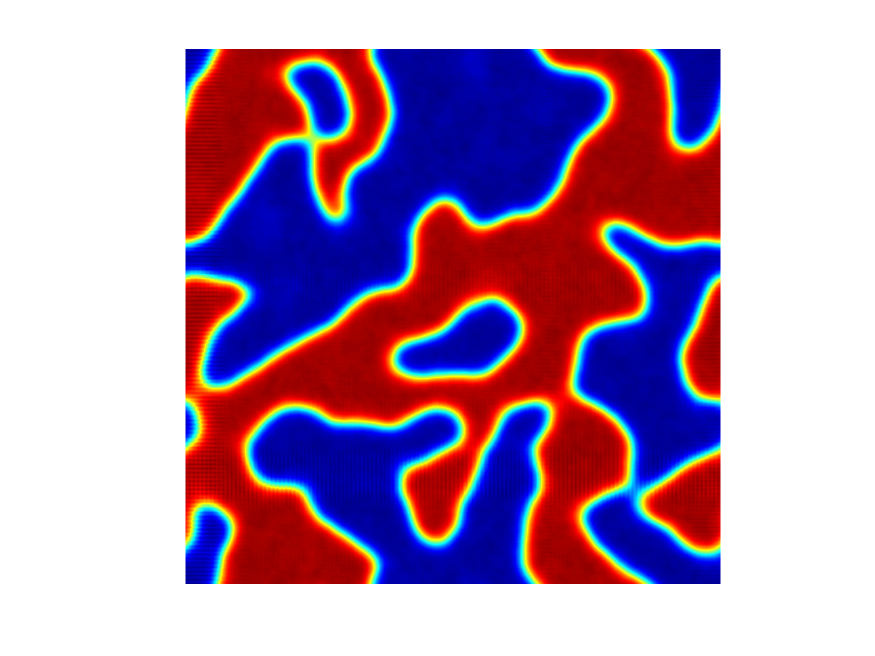} &
\includegraphics[width=0.198\textwidth]{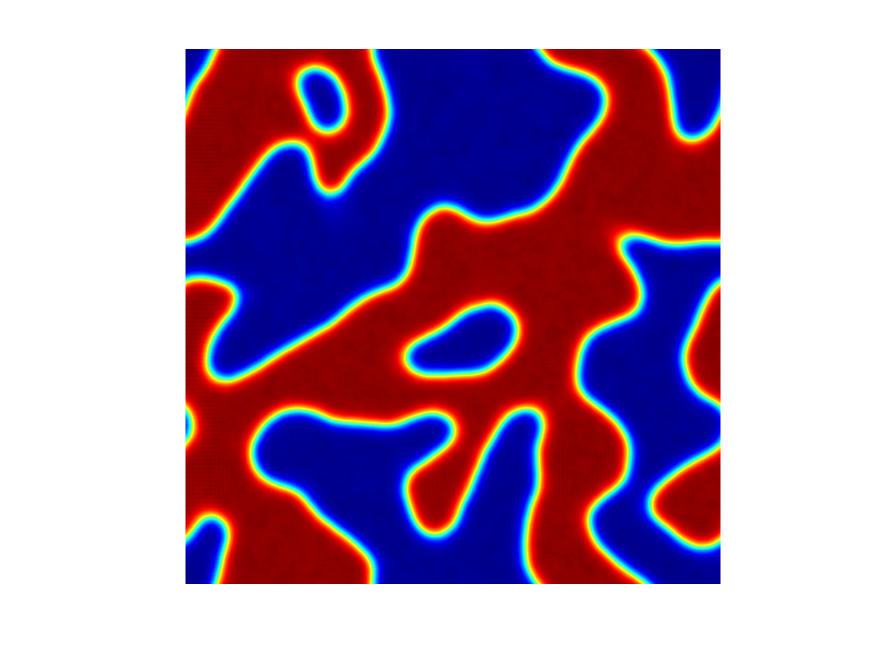} &
\includegraphics[width=0.198\textwidth]{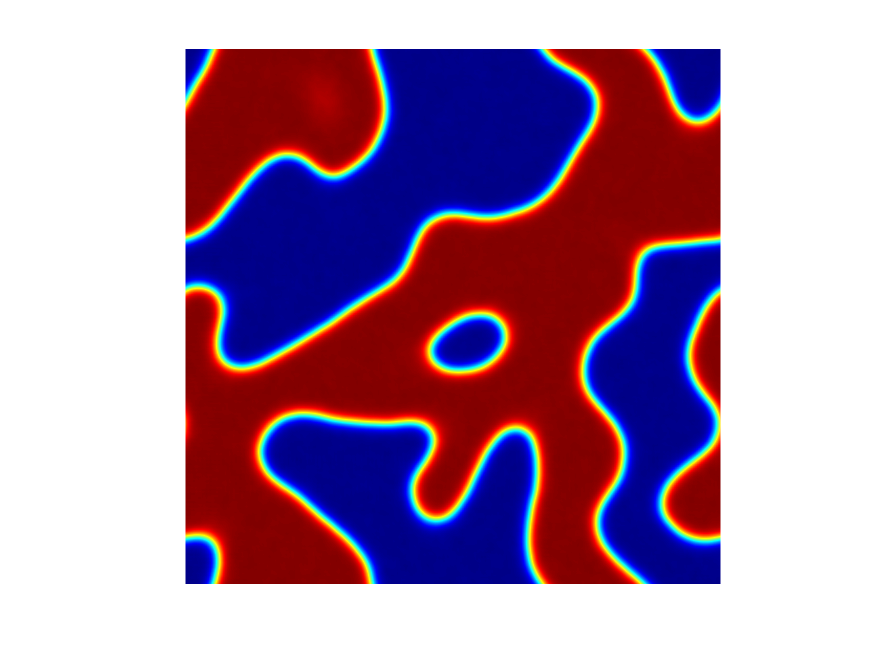} \\

\includegraphics[width=0.198\textwidth]{pic/CD/alpha9999/0_9999.eps} &
\includegraphics[width=0.198\textwidth]{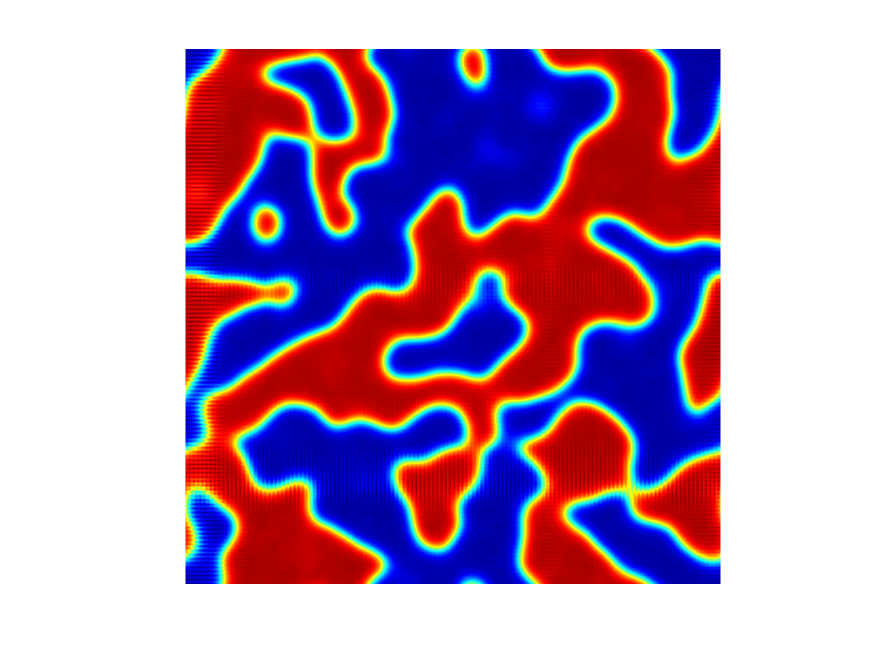} &
\includegraphics[width=0.198\textwidth]{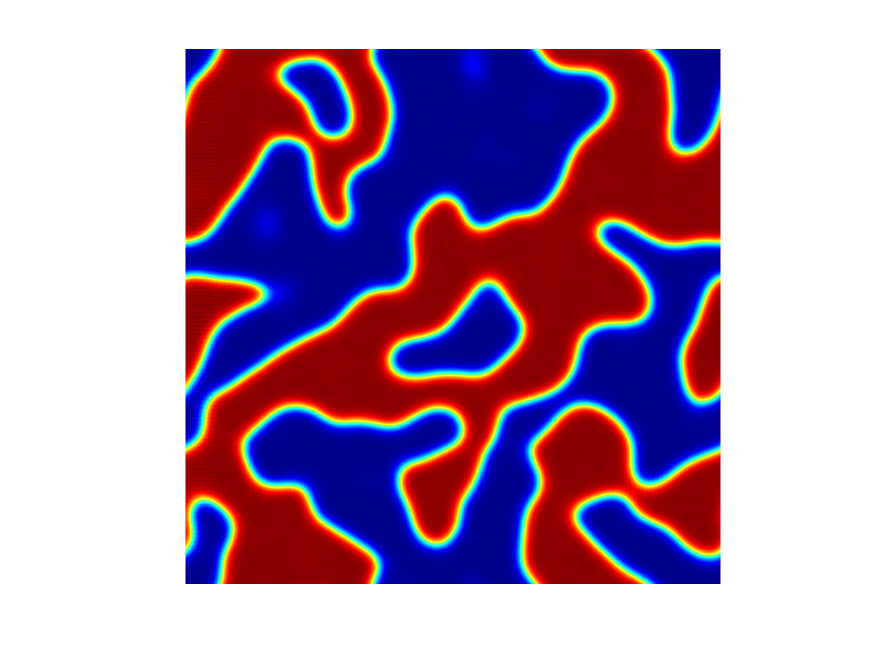} &
\includegraphics[width=0.198\textwidth]{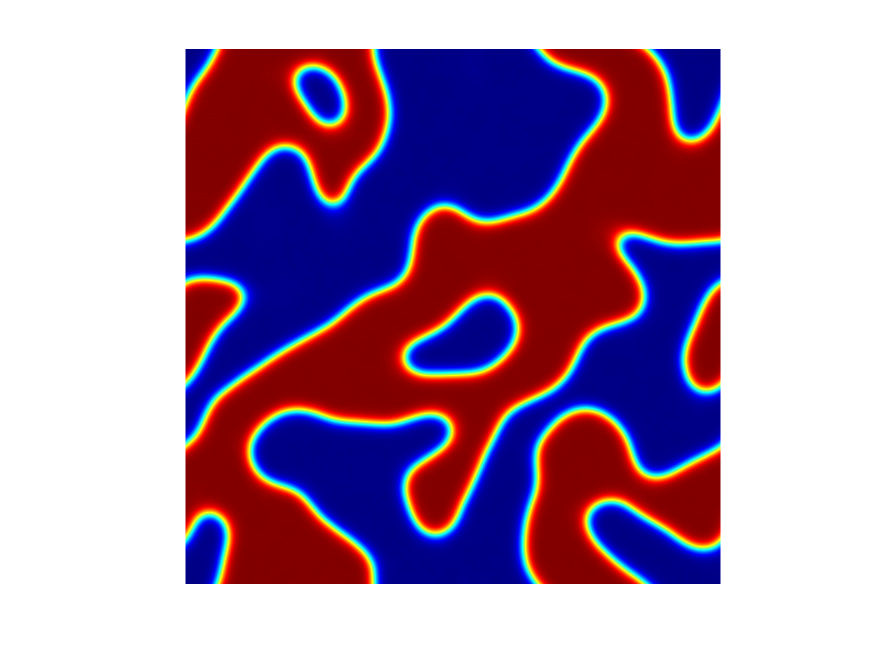} &
\includegraphics[width=0.198\textwidth]{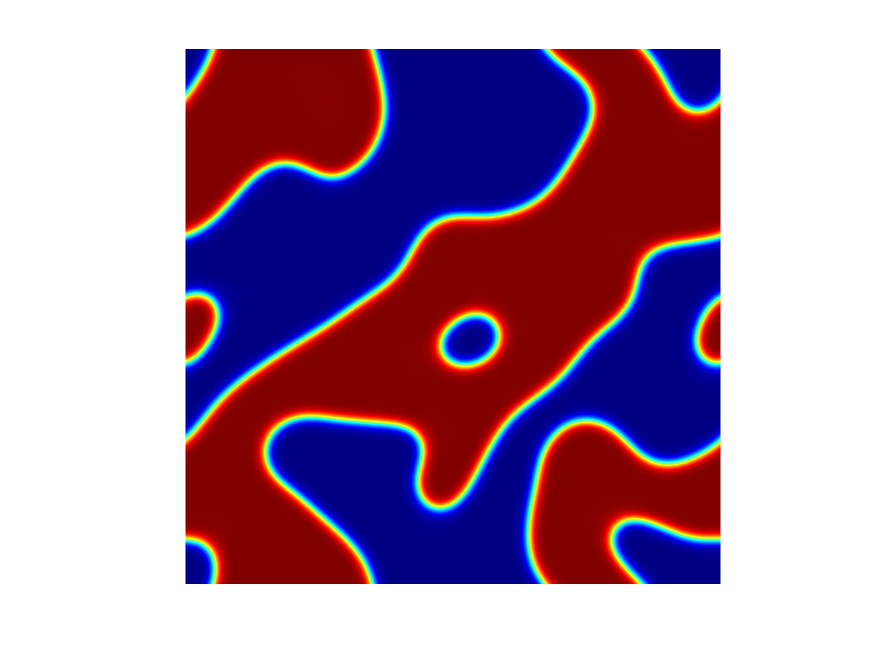} \\

\includegraphics[width=0.198\textwidth]{pic/CD/alpha9999/0_9999.eps} &
\includegraphics[width=0.198\textwidth]{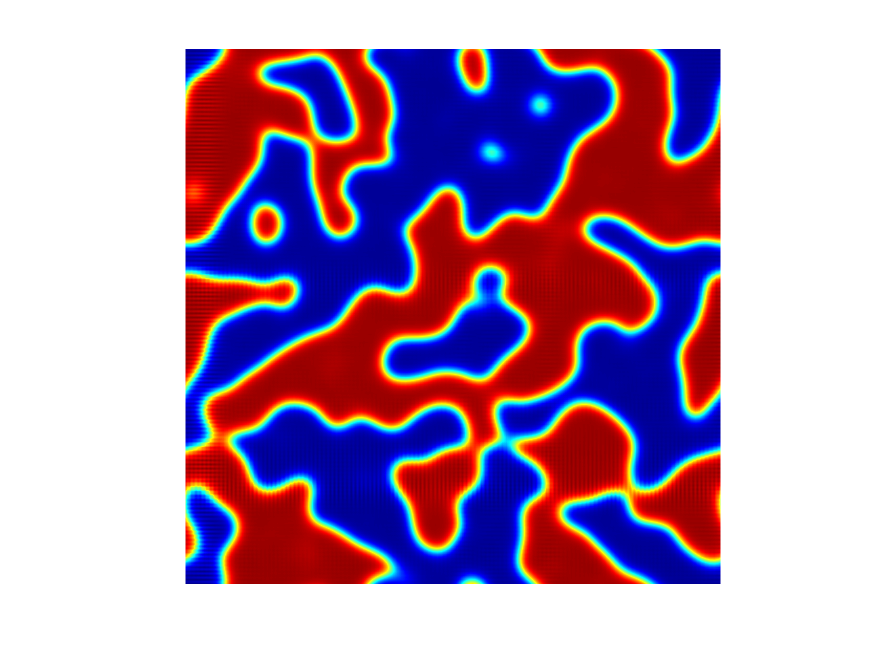} &
\includegraphics[width=0.198\textwidth]{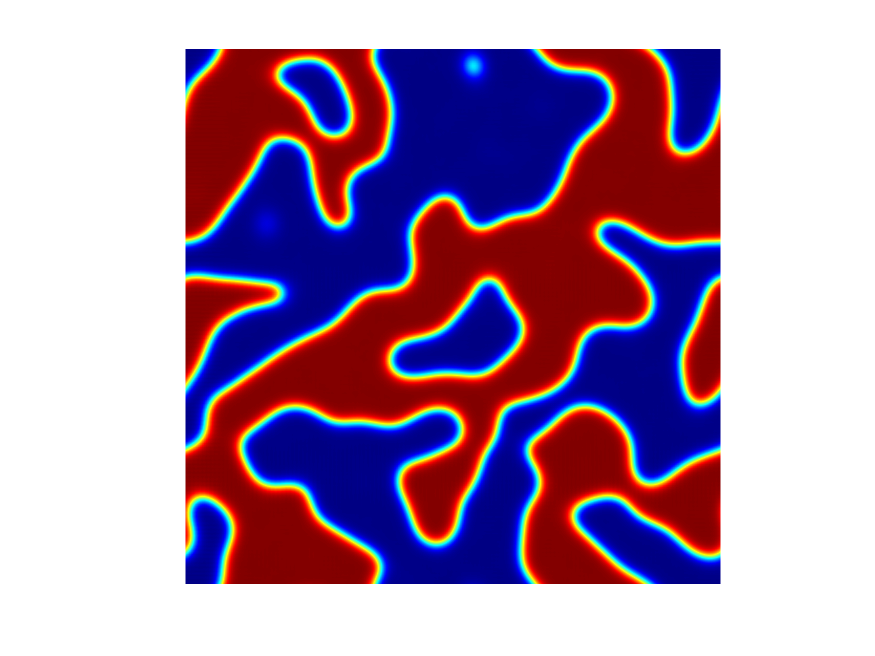} &
\includegraphics[width=0.198\textwidth]{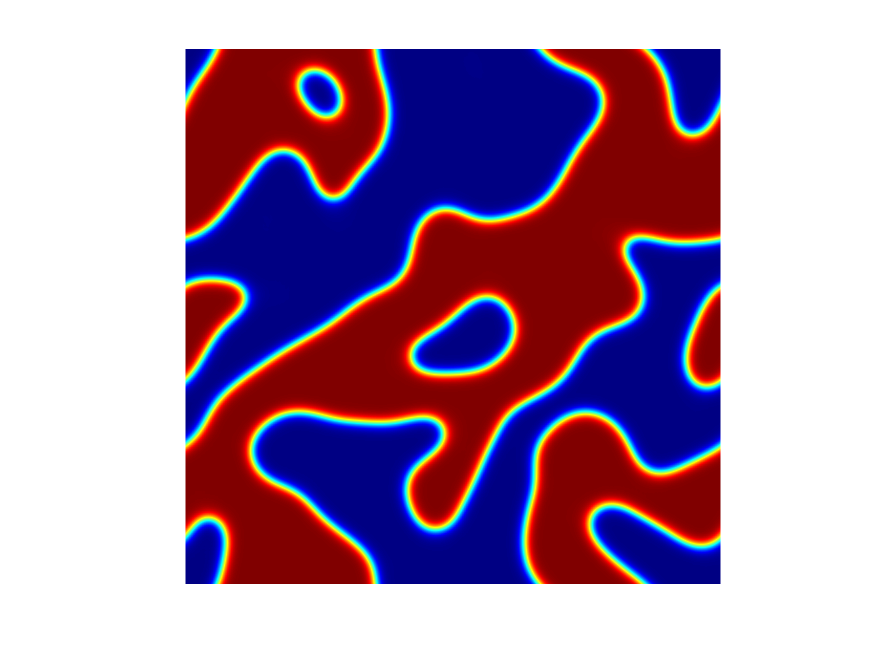} &
\includegraphics[width=0.198\textwidth]{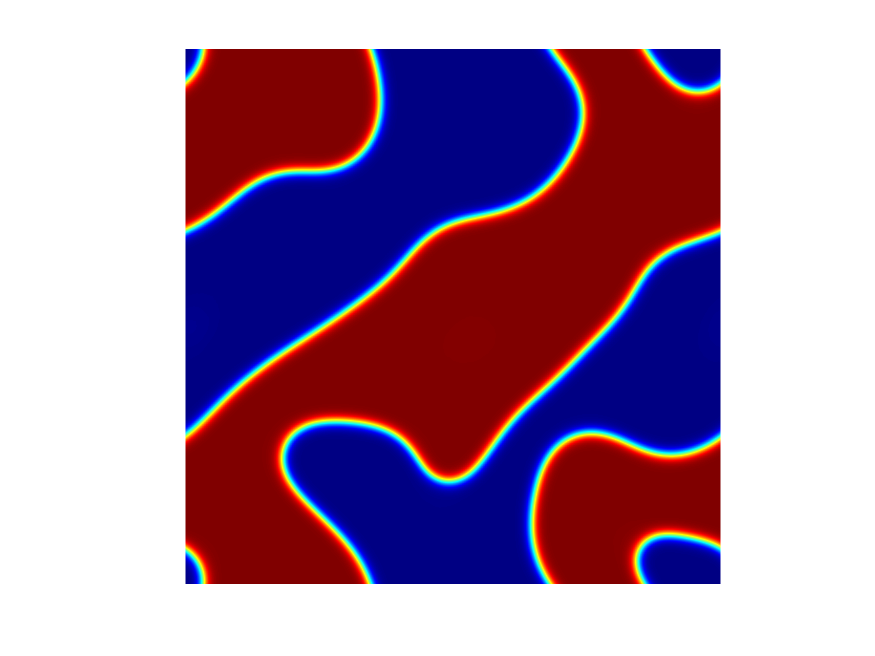} \\

\includegraphics[width=0.198\textwidth]{pic/CD/alpha9999/0_9999.eps} &
\includegraphics[width=0.198\textwidth]{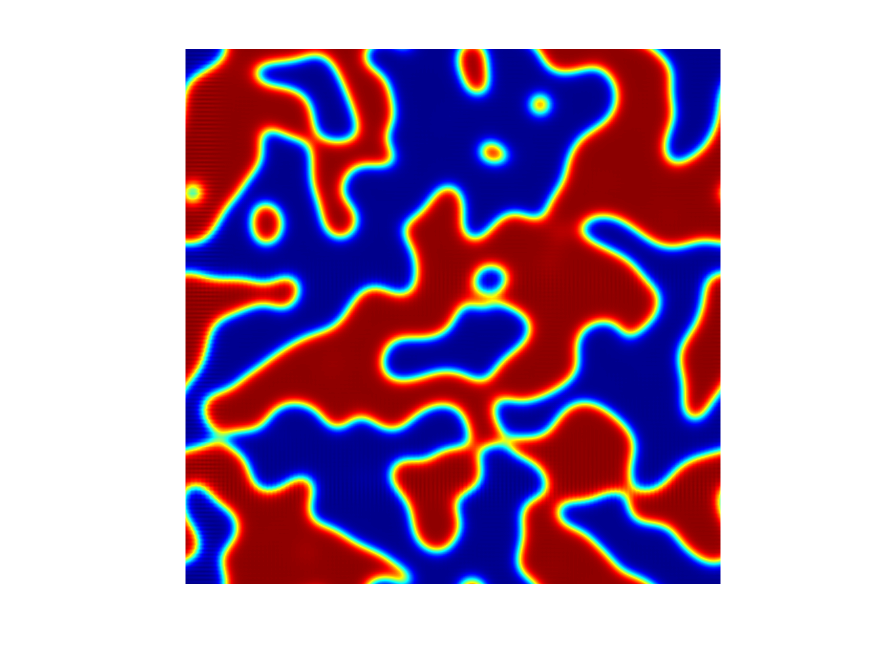} &
\includegraphics[width=0.198\textwidth]{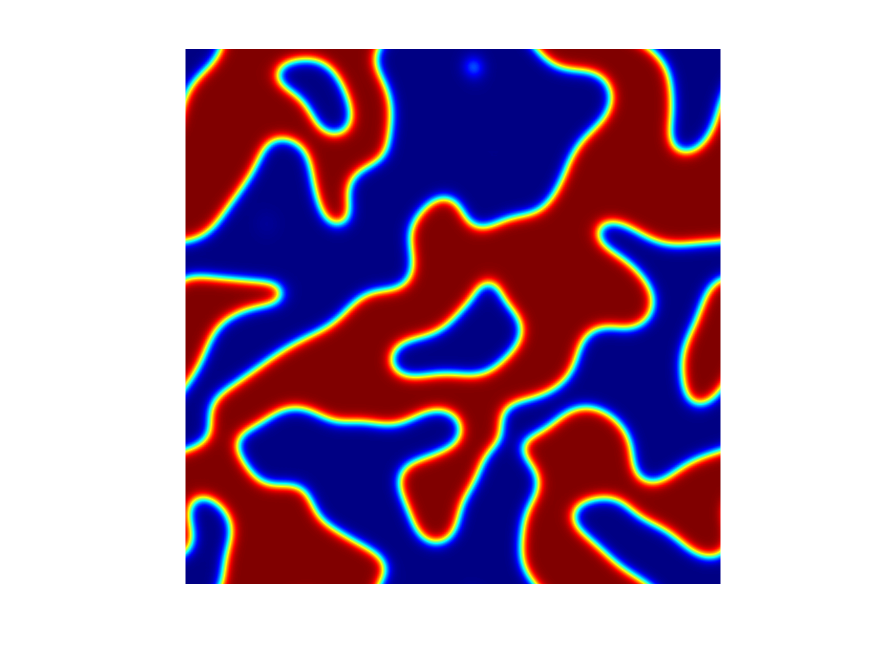} &
\includegraphics[width=0.198\textwidth]{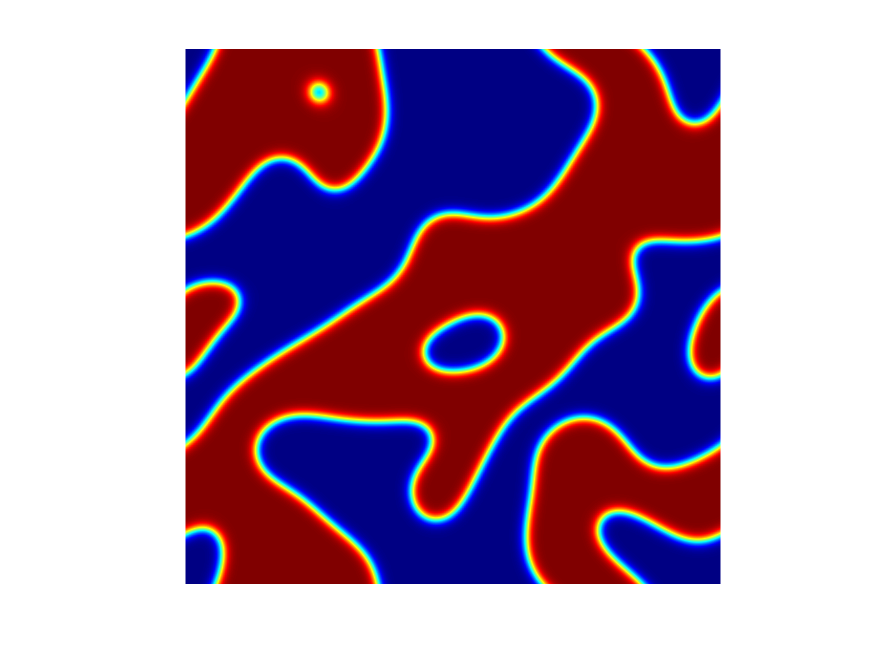} &
\includegraphics[width=0.198\textwidth]{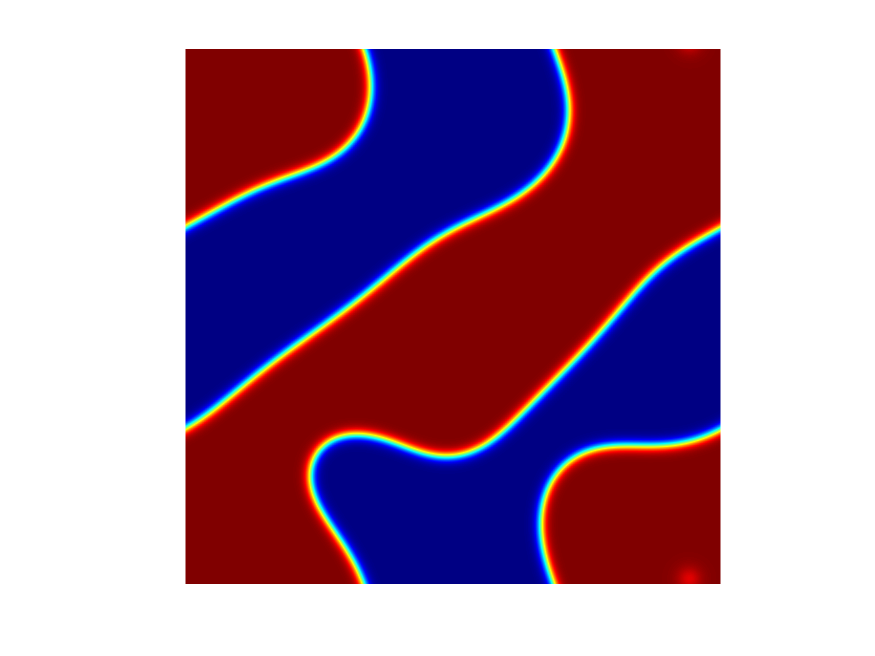} \\

\includegraphics[width=0.198\textwidth]{pic/CD/alpha9999/0_9999.eps} &
\includegraphics[width=0.198\textwidth]{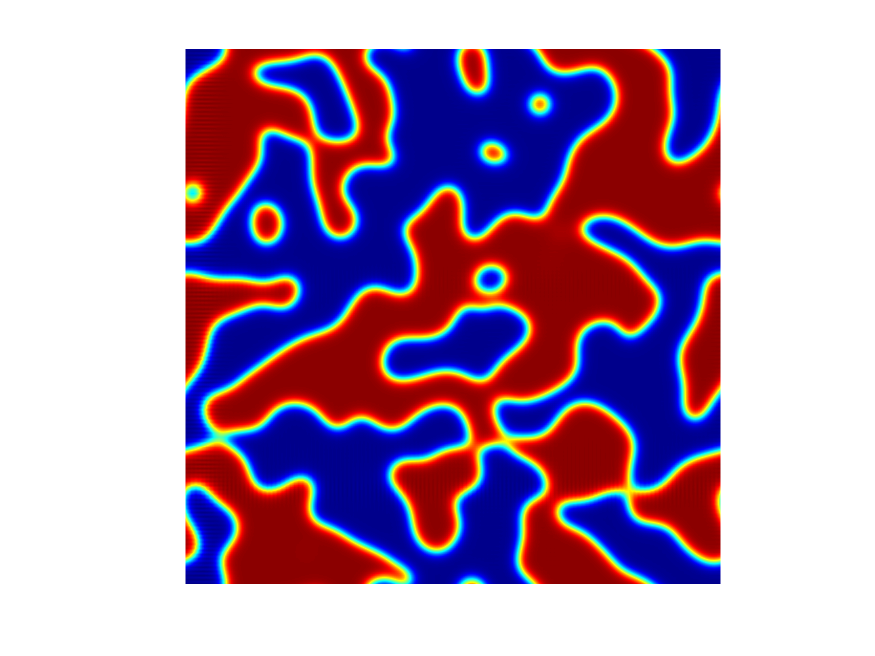} &
\includegraphics[width=0.198\textwidth]{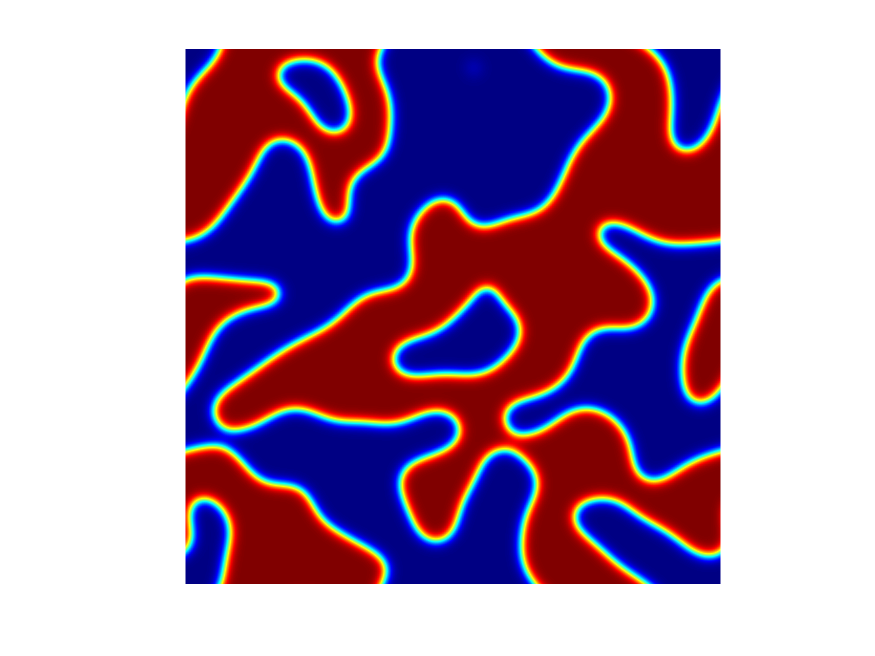} &
\includegraphics[width=0.198\textwidth]{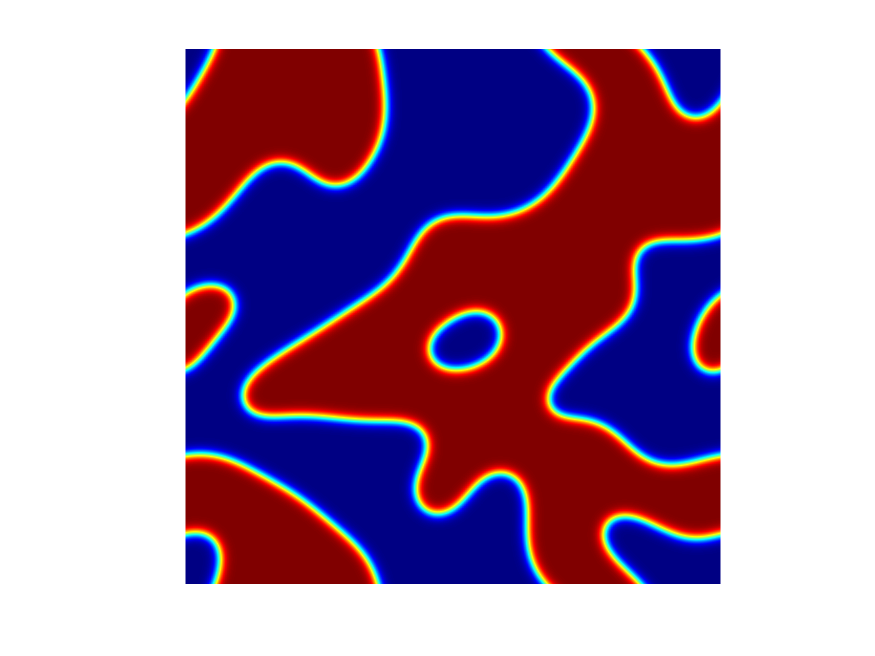} &
\includegraphics[width=0.198\textwidth]{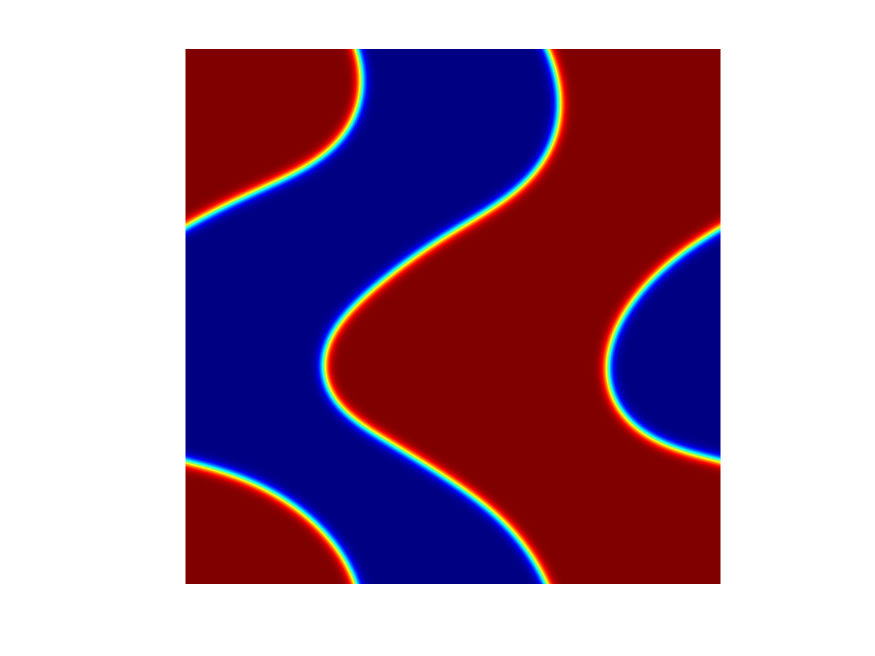}

\end{tabular}
\end{adjustbox}

\caption{Example \ref{example5}. 
Evolution of the solution with different values of fractional orders $\alpha$.
From top to bottom: $\alpha=0.3,0.5,0.7,0.9,0.9999$.
From left to right: $t=0,10,20,50,200$.
}

\label{Coarsening_Dynamics_for_AC}

\end{figure}

    To further quantify the coarsening dynamics, we track two diagnostic quantities: the domain number $N_d(t)$ and
the interface length $P_E(t)$.   The domain number $N_d(t)$ is defined as the
number of connected components of the set $\{\phi > 0\}$. The interface length
is computed via the Ginzburg--Landau functional approximation with
$\varepsilon=0.01$,
\begin{equation*}\label{eq:PE}
  P_E(t) = \int_\Omega \left( \frac{\varepsilon}{2} |\nabla\phi|^2
  + \frac{1}{\varepsilon} F(\phi) \right) d\mathbf{x}.
\end{equation*}

The evolution of $N_d(t)$ and $P_E(t)$ for different values of $\alpha$ is
presented in Figure~\ref{Domain_and_Interface_energy}. In classical Allen--Cahn dynamics ($\alpha \to 1$), the
coarsening is driven entirely by mean curvature flow: small domains shrink and
annihilate rapidly, and both $N_d$ and $P_E$ decrease quickly. As the value of 
$\alpha$ decreases, the coarsening is progressively retarded, and both
quantities decay significantly more slowly.

\begin{figure}[htbp]
	\centering
    \subfigure[The domain number] 
	{
		\begin{minipage}{0.35\linewidth}
			\centering
			\includegraphics[width=\linewidth]{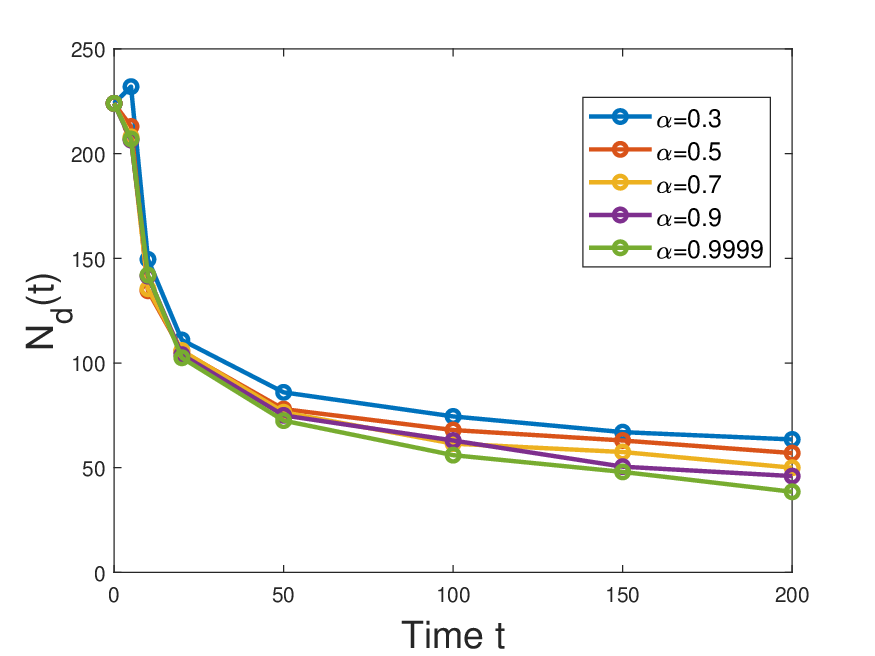}
		\end{minipage}
	}
    \subfigure[The interface length] 
	{
		\begin{minipage}{0.35\linewidth}
			\centering
			\includegraphics[width=\linewidth]{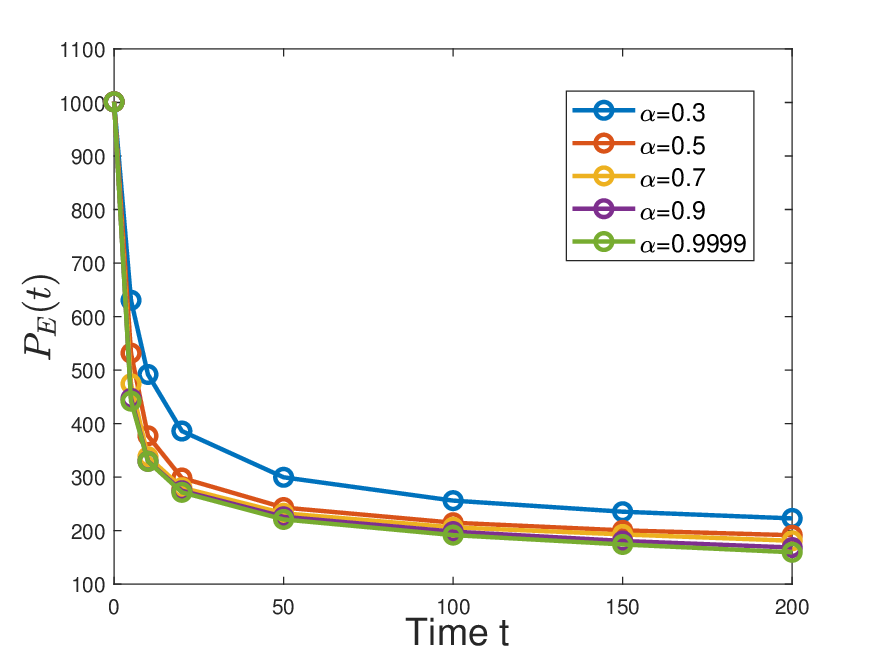}
		\end{minipage}
	}
	
 \caption{Example \ref{example5}. 
Evolution of the domain number $N_d(t)$ (left) and the interface length
$P_E(t)$ (right) for different fractional orders $\alpha$ with $\varepsilon=0.01$.
}
 \label{Domain_and_Interface_energy}
\end{figure}


This behavior is a direct consequence of the friction mechanism identified in 
 \eqref{eq:internal_balance}--\eqref{eq:macro_balance1} (see Remark~\ref{rmk:force-balance}). As $\alpha$ decreases, the relaxation spectrum
$\kappa(\xi) \propto \xi^{\alpha-1}$ shifts weight toward slow modes whose
elastic deformation $\Psi$ persists over the coarsening time scale. By the
internal force balance \eqref{eq:internal_balance}, each unrelaxed mode exerts a friction
force $\kappa(\xi)\Psi$ on $\phi$ that directly opposes the curvature-driven
interface motion. The smaller $\alpha$ is, the greater the aggregate persistent
friction, and the slower the domain annihilation. As $\alpha \to 1$, only
instantaneous (Markovian) friction remains, and classical Allen--Cahn
coarsening is recovered.


\section{Conclusion}\label{sec:7}
This paper develops a thermodynamically consistent modeling and computational framework for phase-field dynamics with memory effects. The central idea is that time-fractional phase-field equations should be interpreted as reduced representations of more fundamental memory-dependent systems rather than as the primary modeling starting point. By incorporating a continuum of relaxation modes into Onsager's variational principle, we establish a class of memory-augmented phase-field (MAPF) equations that naturally account for non-Markovian effects while preserving the gradient-flow structure of classical phase-field models.

The proposed framework provides a direct connection between relaxation spectra, memory kernels, and internal-variable representations. For a power-law relaxation spectrum, the corresponding fractional phase-field model is recovered as a special reduced case. More importantly, the internal relaxation modes reveal the hidden memory variables, stored memory energy, and dissipation mechanisms underlying the fractional formulation. This perspective also indicates that other choices of relaxation spectra or internal constitutive laws can generate a broader class of memory-dependent phase-field models beyond the fractional setting.

To enable efficient computation, we develop an equivalent internal-variable lifting of the MAPF equations, which transforms the original nonlocal-in-time dynamics into an extended Markovian system. Based on this formulation, we establish the continuous energy dissipation law and mass conservation properties and design first- and second-order energy-stable exponential scalar auxiliary variable (ESAV) schemes. By exploiting the rank-one coupling structure induced by the relaxation spectrum representation, we further develop an efficient implementation with reduced computational complexity and storage requirements. Numerical experiments demonstrate that the proposed methods accurately capture long-time memory effects while preserving the essential physical structures of phase-field dynamics.

The present work establishes a general framework that connects memory-dependent phase-field modeling, internal-variable thermodynamics, and non-Markovian dissipative systems. The framework suggests that memory effects in phase-field dynamics can be systematically modeled through relaxation spectra rather than introduced solely through fractional operators, opening new possibilities for constructing physically interpretable models with diverse relaxation behaviors. Future research will investigate experimentally informed relaxation spectra, more general constitutive laws, and efficient computational techniques for complex memory-driven phase-field systems.

\section*{Acknowledgments}


This work was supported by the National Key R\&D Program of China (Grant No.: 2023YFA1011403), the NSFC (Grant No.: 12671515, 92470112 and 12426312), the Zhejiang Provincial Natural Science Foundation of China (Grant No.: LR26A010001), and the Shenzhen Science and Technology Program (Grant No.: QNXMA20250701095429040).

\bibliographystyle{elsarticle-num}
\bibliography{references_memory_modeling_polished}

@article{HouXu2021,
author = {Hou, Dianming and Xu, Chuanju},
title = {Highly Efficient and Energy Dissipative Schemes for the Time Fractional {Allen-Cahn} Equation},
journal = {SIAM J. Sci. Comput.},
volume = {43},
number = {5},
pages = {A3305-A3327},
year = {2021},
}

@article{QuanTangYang2020,
author = {Quan, Chaoyu and Tang, Tao and Yang, Jiang},
title = {How to define dissipation-preserving energy for time-fractional phase-field equations},
journal = {CSIAM Trans. Appl. Math.},
volume = {1},
number = {3},
pages = {478-490},
year = {2020},
}

@article{QuanTangYang2023,
author = {Quan, Chaoyu and Tang, Tao and Wang, Boyi and Yang, Jiang},
title = {A decreasing upper bound of the energy for time-fractional
 phase-field equations},
journal = {Commun. Comput. Phys.},
volume = {33},
pages = {962-991},
year = {2023},
}

@article{ShenYang2010,
title = {Numerical approximations of {Allen-Cahn} and {Cahn-Hilliard} equations},
journal = {Discrete Contin.Dyn. Syst.,Ser. A},
volume = {28},
pages = {1669-1691},
year = {2010},
author = {Shen, Jie and Yang, Jiang},
}

@article{SHEN2019474,
title = {A New Class of Efficient and Robust Energy Stable Schemes
 for Gradient Flows},
journal = {SIAM Rev.},
volume = {61},
pages = {474-506},
year = {2019},
author = {Shen, Jie and Xu, Jie and Yang, Jiang},
}

@article{LiaoTao2021,
author = {Liao, Honglin and Tang, Tao and Zhou, Tao},
title = {An Energy Stable and Maximum Bound Preserving Scheme with Variable Time Steps for Time Fractional {Allen-Cahn} Equation},
journal = {SIAM J. Sci. Comput.},
volume = {43},
number = {5},
pages = {A3503-A3526},
year = {2021},
}

@article{Fritz2021,
title = {Equivalence between a time-fractional and an integer-order gradient flow: {The} memory effect reflected in the energy},
author = {Fritz, Marvin  and  Khristenko, Ustim  and  Wohlmuth, Barbara},
pages = {20220262},
volume = {12},
journal = {Adv. Nonlinear Anal.},
year = {2023},
}

@article{DingMao2024,
  title={A novel class of arbitrary high-order numerical schemes for  fractional differential equations},
  author={Ding,Peng and Mao, Zhiping},
journal  = {J. Sci. Comput.},
volume = {105},
pages = {82},
year = {2025},
}

@book{QiangDu,
author = {Du, Qiang},
title = {Nonlocal Modeling, Analysis, and Computation},
publisher = {Society for Industrial and Applied Mathematics},
year = {2019},
address = {Philadelphia, PA},
}

@article{Tang_PF,
author = {Tang, Tao and Yu, Haijun and Zhou, Tao},
title = {On Energy Dissipation Theory and Numerical Stability for Time-Fractional Phase-Field Equations},
journal = {SIAM J. Sci. Comput.},
volume = {41},
number = {6},
pages = {A3757-A3778},
year = {2019},
}

@article{sESAV_2022,
    author = {Ju, Lili  and Li, Xiao and  Qiao, Zhonghua},
    title = {Stabilized Exponential-{SAV} Schemes Preserving Energy Dissipation Law and Maximum Bound Principle for The {Allen–Cahn} Type Equations},
    journal = {J. Sci. Comput},
    volume = {92},
    pages = {66},
    year = {2022},
}

@article{DuYangZhou2020JSC,
  author  = {Qiang Du and Jiang Yang and Zhi Zhou},
  title   = {Time-Fractional Allen--Cahn Equations: Analysis and Numerical Methods},
  journal = {J. Sci. Comput},
  year    = {2020},
  volume  = {85},
  number  = {2},
  pages   = {1--30},
  note    = {Article 42},
  doi     = {10.1007/s10915-020-01351-5}
}

@article{WinterChambon1986,
  author  = {Winter, H. Henning and Chambon, Francois},
  title   = {Analysis of linear viscoelasticity of a crosslinking polymer at the gel point},
  journal = {Journal of Rheology},
  volume  = {30},
  number  = {2},
  pages   = {367--382},
  year    = {1986},
  doi     = {10.1122/1.549853}
}

@article{Fabry2003,
  author  = {Fabry, Ben and Maksym, Geoffrey N. and Butler, James P. and Glogauer, Michael and Navajas, Daniel and Fredberg, Jeffrey J.},
  title   = {Time scale and other invariants of integrative mechanical behavior in living cells},
  journal = {Physical Review E},
  volume  = {68},
  number  = {4},
  pages   = {041914},
  year    = {2003},
  doi     = {10.1103/PhysRevE.68.041914}
}

@article{Qian2026Incorporating,
  author  = {Qian, Tiezheng and Yu, Wenkai and Zhang, Zhen and Zhang, Rui},
  title   = { Incorporating convected time derivatives into Onsager's variational principle: viscoelastic fluids and nematic liquid crystals},
  journal = {Chinese Physics B},
  year    = {2026},
  volume  = {},
  pages   = {},
}

@article{AlMaskariKaraa2021IMAJNA,
  author  = {M. Al-Maskari and S. Karaa},
  title   = {The time-fractional Cahn--Hilliard equation: analysis and approximation},
  journal = {IMA Journal of Numerical Analysis},
  year    = {2021},
  volume  = {42},
  number  = {2},
  pages   = {1831--1865},
  doi     = {10.1093/imanum/drab025}
}

@book{Ottinger2005BET,
  author    = {Hans Christian {\"O}ttinger},
  title     = {Beyond Equilibrium Thermodynamics},
  publisher = {Wiley-Interscience},
  year      = {2005},
}

@article{diethelm2022new,
  title={A New Diffusive Representation for Fractional Derivatives and its Application},
  author={Diethelm, Kai},
  journal={arXiv preprint arXiv:2204.04023},
  year={2022}
}

@Article{Bonfanti2020,
author ={Bonfanti, A. and Kaplan, J. L. and Charras, G. and Kabla, A.},
title  ={Fractional viscoelastic models for power-law materials},
journal  ={Soft Matter},
year  ={2020},
volume  ={16},
number  ={26},
pages  ={6002-6020},
}

@article{DingMaoLiu,
title = {Stability and Error Estimates of {BDF-k} Numerical Schemes for
Sub-diffusion Equation with Tempered Fractional Derivative},
journal = {J. Sci. Comput.},
volume = {107},
pages = {85},
year = {2026},
author = {Peng Ding and Qingxia Liu and Zhiping Mao},
}

@article{Montseny1998,
title = {Diffusive representation of pseudo-differential time-operators},
journal = {ESAIM: Proc.},
volume = {5},
pages = {159-175},
year = {1998},
author = {
Gerard Montseny },
}

@article{Gurtin1996Microforce,
  author = {Gurtin, Morton E.},
  title = {Generalized {Ginzburg--Landau} and {Cahn--Hilliard} Equations Based on a Microforce Balance},
  journal = {Physica D: Nonlinear Phenomena},
  volume = {92},
  number = {3--4},
  pages = {178--192},
  year = {1996},
  doi = {10.1016/0167-2789(95)00173-5}
}

@article{ColemanGurtin1967InternalVariables,
  author = {Coleman, Bernard D. and Gurtin, Morton E.},
  title = {Thermodynamics with Internal State Variables},
  journal = {The Journal of Chemical Physics},
  volume = {47},
  number = {2},
  pages = {597--613},
  year = {1967},
  doi = {10.1063/1.1711937}
}

@article{RotsteinBrandonNovickCohenNepomnyashchy2001Memory,
  author = {Rotstein, Horacio G. and Brandon, Simon and Novick-Cohen, Amy and Nepomnyashchy, Alexander},
  title = {Phase Field Equations with Memory: The Hyperbolic Case},
  journal = {SIAM Journal on Applied Mathematics},
  volume = {62},
  number = {1},
  pages = {264--282},
  year = {2001},
  doi = {10.1137/S0036139900369102}
}

@article{Tanaka2000VPS,
  author = {Tanaka, Hajime},
  title = {Viscoelastic Phase Separation},
  journal = {Journal of Physics: Condensed Matter},
  volume = {12},
  number = {15},
  pages = {R207--R264},
  year = {2000},
  doi = {10.1088/0953-8984/12/15/201}
}

@article{Tanaka2022BioCells,
  author = {Tanaka, Hajime},
  title = {Viscoelastic Phase Separation in Biological Cells},
  journal = {Communications Physics},
  volume = {5},
  pages = {167},
  year = {2022},
  doi = {10.1038/s42005-022-00947-7}
}

@article{CurkLuijten2023PNAS,
  author = {Curk, Tine and Luijten, Erik},
  title = {Phase Separation and Ripening in a Viscoelastic Gel},
  journal = {Proceedings of the National Academy of Sciences of the United States of America},
  volume = {120},
  number = {32},
  pages = {e2304655120},
  year = {2023},
  doi = {10.1073/pnas.2304655120}
}

@article{LiuChengWangZhao2018TFPF,
  author = {Liu, Huan and Cheng, Aijie and Wang, Hong and Zhao, Jia},
  title = {Time-Fractional {Allen--Cahn} and {Cahn--Hilliard} Phase-Field Models and Their Numerical Investigation},
  journal = {Computers \& Mathematics with Applications},
  volume = {76},
  number = {8},
  pages = {1876--1892},
  year = {2018},
  doi = {10.1016/j.camwa.2018.07.036}
}

@article{LiQuanXu2022Energy,
  author = {Li, Dong and Quan, Chaoyu and Xu, Jiao},
  title = {Energy-Dissipation for Time-Fractional Phase-Field Equations},
  journal = {Communications on Pure and Applied Analysis},
  volume = {21},
  number = {10},
  pages = {3371--3387},
  year = {2022},
  doi = {10.3934/cpaa.2022104}
}

@book{Widder1941,
  author    = {Widder, David Vernon},
  title     = {The {L}aplace Transform},
  publisher = {Princeton University Press},
  address   = {Princeton, NJ},
  year      = {1941},
}

@book{Schilling2012,
  author    = {Schilling, Ren{\'e} L. and Song, Renming and Vondra{\v{c}}ek, Zoran},
  title     = {Bernstein Functions: Theory and Applications},
  edition   = {2nd},
  series    = {De Gruyter Studies in Mathematics},
  volume    = {37},
  publisher = {De Gruyter},
  address   = {Berlin},
  year      = {2012},
}

@article{DoiOnuki1992,
  author  = {M. Doi and A. Onuki},
  title   = {Dynamic coupling between stress and composition in polymer solutions and
     blends},
  journal = {J. Phys. II France.},
  volume  = {2},
  number = {8},
  pages   = {1631--1656},
  year    = {1992},
}

@article{Tanaka1997,
  author  = {Tanaka, Hajime},
  title   = {Viscoelastic phase separation in soft matter and foods},
  journal = {Faraday Discussions.},
  volume  = {156},
  pages   = {371--406},
  year    = {2012},
}

@article{deGennes1976,
  author  = {P. G. de Gennes},
  title   = {On a relation between percolation theory and the elasticity of gels},
  journal = {J. Physique Lett.},
  volume  = {37},
   number  = {1},
  pages   = {L1-L2},
  year    = {1976},
}

@article{Muthukumar1989,
  author  = {M. Muthukumar},
  title   = {Screening effect on viscoelasticity near the gel point},
  journal = {Macromolecules.},
  volume  = {22},
  number  = {12},
  pages   = {4656--4658},
  year    = {1989},
}

@article{AllenCahn1979,
  author = {Allen, Samuel M. and Cahn, John W.},
  title = {A Microscopic Theory for Antiphase Boundary Motion and Its Application to Antiphase Domain Coarsening},
  journal = {Acta. Metall.},
  volume = {27},
  number = {6},
  pages = {1085-1095},
  year = {1979},
}

@article{CahnHilliard1958,
  author = {Cahn, John W. and Hilliard, John E.},
  title = {Free Energy of a Nonuniform System. {I}. {Interfacial} Free Energy},
  journal = {J. Chem. Phys.},
  volume = {28},
  number = {2},
  pages = {258-267},
  year = {1958},
}

@article{LiBianLiKarniadakis2015MZ,
  author = {Li, Zhen and Bian, Xin and Li, Xiantao and Karniadakis, George Em},
  title = {Incorporation of Memory Effects in Coarse-Grained Modeling via the {Mori--Zwanzig} Formalism},
  journal = {J. Chem. Phys.},
  volume = {143},
  number = {24},
  pages = {243128},
  year = {2015},
}

@article{sherman1950adjustment,
  title={Adjustment of an inverse matrix corresponding to a change in one element of a given matrix},
  author={Sherman, Jack and Morrison, Winifred J.},
  journal={Ann. Math. Statist.},
  volume={21},
  number={1},
  pages={124--127},
  year={1950},
}

@book{Abramowitz1965,
  title={Handbook of Mathematical Functions with Formulas, Graphs, and Mathematical Tables},
  author={Abramowitz, Milton and Stegun, Irene A.},
  year={1965},
  publisher={Dover Publications}
}

@article{sabzikar2015tempered,
  title={Tempered fractional calculus},
  author={Sabzikar, Farzad and Meerschaert, Mark M. and Chen, Jinghua},
  journal={J. Comput. Phys.},
  volume={293},
  pages={14--28},
  year={2015},
}

@article{song2023nonmaxwellian,
  title={Non-Maxwellian viscoelastic stress relaxations in soft matter},
  author={Song, Jake and Holten-Andersen, Niels and McKinley, Gareth H.},
  journal={Soft Matter.},
  volume={19},
  number={41},
  pages={7885--7906},
  year={2023},
}

@article{mainardi2008time,
  title={Time-fractional diffusion of distributed order},
  author={Mainardi, Francesco and Mura, Antonio and Pagnini, Gianni and Gorenflo, Rudolf},
  journal={J. Vib. Control.},
  volume={14},
  number={9--10},
  pages={1267--1290},
  year={2008},
}

@article{luo2018development,
  title={Development of Prony series models based on continuous relaxation spectrums for relaxation moduli determined using creep tests},
  author={Luo, Rong and Lv, Huijie and Liu, Hanqi},
  journal={Constr. Build. Mater.},
  volume={168},
  pages={758--770},
  year={2018},
}

@book{Ferry1980,
  author    = {Ferry, John D.},
  title     = {Viscoelastic Properties of Polymers},
  edition   = {3rd},
  publisher = {Wiley},
  address   = {New York},
  year      = {1980}
}

@book{Tschoegl1989,
  author    = {Tschoegl, Nicholas W.},
  title     = {The Phenomenological Theory of Linear Viscoelastic Behavior: An Introduction},
  publisher = {Springer},
  address   = {Berlin},
  year      = {1989}
}

@book{Mainardi2010Fractional,
  author    = {Mainardi, Francesco},
  title     = {Fractional Calculus and Waves in Linear Viscoelasticity},
  publisher = {Imperial College Press},
  address   = {London},
  year      = {2010}
}

@article{Bagley1983Fractional,
  author  = {Bagley, Ronald L. and Torvik, Paul J.},
  title   = {A Theoretical Basis for the Application of Fractional Calculus to Viscoelasticity},
  journal = {J. Rheol.},
  volume  = {27},
  number  = {3},
  pages   = {201--210},
  year    = {1983}
}

@book{Kilbas2006Theory,
  author    = {Kilbas, Anatoly A. and Srivastava, Hari M. and Trujillo, Juan J.},
  title     = {Theory and Applications of Fractional Differential Equations},
  publisher = {Elsevier},
  year      = {2006}
}
\end{document}